\documentclass{article}
\usepackage{graphicx}
\usepackage{arxiv}
\usepackage{natbib}

\usepackage{amsthm}
\usepackage{circledsteps}
\usepackage{amssymb, amsmath, latexsym}
\usepackage{url}
\usepackage{algorithm}
\usepackage{algorithmic}
\usepackage{tabularx}
\usepackage{paralist}
\usepackage{mathtools}
\usepackage{xcolor}

\usepackage{bbm} 

\usepackage{makecell}
\usepackage{multirow}
\usepackage{booktabs}

\usepackage{nicefrac}       

\usepackage[flushleft]{threeparttable} 

\usepackage{caption}
\usepackage{multirow}
\usepackage{colortbl}
\definecolor{bgcolor}{rgb}{0.8,1,1}
\definecolor{bgcolor2}{rgb}{0.8,1,0.8}
\definecolor{niceblue}{rgb}{0.0,0.19,0.56}

\usepackage{hyperref}
\hypersetup{colorlinks,linkcolor={blue},citecolor={niceblue},urlcolor={blue}}

\newtheorem{assumption}{Assumption}

\usepackage{pifont}
\definecolor{PineGreen}{RGB}{0,110,51}
\definecolor{BrickRed}{RGB}{143,20,2}

\def\nikita#1{{\color{black}#1}}

\usepackage{subcaption}

\newcommand{\R}{\mathbb{R}}

\def\<#1,#2>{\left\langle #1,#2\right\rangle}

\newtheorem{lemma}{Lemma}[section]
\newtheorem{theorem}{Theorem}[section]

\newtheorem{proposition}{Proposition}[section]
\newtheorem{corollary}{Corollary}[section]
\newtheorem{remark}{Remark}[section]
\theoremstyle{plain}

\usepackage{xspace}

\usepackage[colorinlistoftodos,bordercolor=orange,backgroundcolor=orange!20,linecolor=orange,textsize=scriptsize]{todonotes}

\newcommand{\Tr}{\mathrm{Tr}}

\usepackage{hyperref}
\graphicspath{{../plots/}}

\usepackage{makecell}

\usepackage{accents}
\newlength{\dhatheight}

\usepackage{enumitem}

\begin{document}

\title{Intermediate Gradient Methods with Relative Inexactness}


\author{Nikita Kornilov \\
            Moscow Institute of Physics and Technology \\
Skolkovo Institute of Science and Technology \\
             \texttt{kornilov.nm@phystech.edu} \\  
         \And
           Eduard Gorbunov \\
            MBZUAI \\
              \texttt{eduard.gorbunov@mbzuai.ac.ae}
              \And
             Mohammad Alkousa \\
           Moscow Institute of Physics and Technology \\
           \texttt{mohammad.alkousa@phystech.edu}
              \And
              Fedor Stonyakin  \\
Moscow Institute of Physics and Technology \\
V.~Vernadsky Crimean Federal University\\
\texttt{fedyor@mail.ru}\\
         \And
              Pavel Dvurechensky  \\
             Weierstrass Institute for Applied Analysis and Stochastics\\
            \texttt{pavel.dvurechensky@wias-berlin.de} \\
             \And
          Alexander Gasnikov \\
             Moscow Institute of Physics and Technology \\
Skolkovo Institute of Science and Technology \\
Institute for Information Transmission Problems RAS \\
             \texttt{gasnikov@yandex.ru}
}


\maketitle

\begin{abstract}
This paper is devoted to first-order algorithms for smooth convex optimization with inexact gradients. Unlike the majority of the literature on this topic, we consider the setting of relative rather than absolute inexactness. More precisely, we assume that an additive error in the gradient is proportional to the gradient norm, rather than being globally bounded by some small quantity. We propose a novel analysis of the accelerated gradient method under relative inexactness and strong convexity and improve the bound on the maximum admissible error that preserves the linear convergence of the algorithm. In other words, we analyze how robust is the accelerated gradient method to the relative inexactness of the gradient information. Moreover, based on the Performance Estimation Problem (PEP) technique, we show that the obtained result is optimal for the family of accelerated algorithms we consider. Motivated by the existing intermediate methods with absolute error, i.e., the methods with convergence rates that interpolate between slower but more robust non-accelerated algorithms and faster, but less robust accelerated algorithms, we propose an adaptive variant of the intermediate gradient method with relative error in the gradient.

\end{abstract}




\section{Introduction}

Motivated by large-scale optimization problems in machine learning and inverse problems, we focus in this paper on first-order algorithms for smooth convex optimization. In many situations, such algorithms cannot use exact first-order information, i.e., gradients, since it is not available. The standard example of such a situation is stochastic optimization problems \cite{shapiro2021lectures} when only a noisy stochastic approximation of the gradient is available. Another example, which we are focusing on in this paper, is when in a deterministic problem some deterministic error is present in the gradient and function values. For example, this may happen when another problem has to be solved to evaluate the gradient or objective, and this problem cannot be solved exactly due to its complexity. A particular setting of such a situation is given by PDE-constrained optimization problems \cite{baraldi2023proximal,hintermuller2020convexity,matyukhin2021convex}, in which the evaluation of the gradient requires solving a system of direct and adjoint PDEs. Another example is bilevel optimization when a constraint in the upper-level problem is given by a solution to a lower-level problem \cite{sabach2017first,solodov2007explicit}. Inexact gradients also typically arise in optimal control problems and inverse problems, where one needs to solve ODEs or PDEs to find the gradient of the objective function \cite{matyukhin2021convex}. For further details on inexactness in the gradients, we refer to \cite{devolder2013intermediate,polyak1987introduction,stonyakin2021inexact,vasin2023accelerated} and references therein. All these applications motivate the study of first-order algorithms with inexact information. We would like to emphasize the important role that Boris Polyak's book \cite{polyak1987introduction} plays in the analysis of non-accelerated inexact gradient methods (for gradient-free methods see another book co-authored by B.~Polyak  \cite{granichin2003randomizirovannye}). For accelerated methods, the first tight analysis of the convergence of the conjugate gradient method with inexact information was published in \cite{nemirovski1986regularizing}. For the general Nesterov accelerated gradient method the first tight analysis and new theoretically very useful concept of (adversarial) inexactness in the gradient was proposed in \cite{devolder2014first} (see also \cite{d2008smooth}).  

A well-developed branch of research on gradient methods with adversarial errors in the gradient is the study of algorithms with bounded additive error \cite{cohen2018acceleration,d2008smooth,gorbunov2019optimal,khanh2023inexact,polyak1987introduction}. The main result here is as follows: if we consider smooth convex optimization problem on a compact set \cite{d2008smooth,cohen2018acceleration} or use a proper stopping rule criteria \cite{gorbunov2019optimal,vasin2023accelerated}, then non-accelerated and accelerated methods do not accumulate an additive error and we can reach the objective function residual proportional to the level of this error. We mention here also the undeservedly little-known old result of Boris Polyak, that without a proper stopping rule on a whole space, even simple gradient descent can diverge \cite{poljak1981iterative}. 

A less-studied setting is when additive error is proportional to the norm of the gradient, which we refer to as relative error (or relative noise) from B.~Polyak's book \cite{polyak1987introduction}:
 $$\|\widetilde{\nabla} f(x) - \nabla f(x)\|_2 \leq \hat{\varepsilon} \|\nabla f(x)\|_2, \quad \forall \hat{\varepsilon} \in [0, 1].$$
Despite the result for the gradient method being classical \cite{polyak1987introduction}, the analysis of accelerated gradient methods is quite challenging in this setting. The recent works \cite{gannot2022frequency,vasin2023accelerated} 
analyze accelerated gradient methods with relative error in the strongly convex case. 
Recently, the analysis of gradient-type methods with relative error was obtained as a by-product of the developments in stochastic optimization with decision-dependent distribution \cite{drusvyatskiy2023stochastic} and policy evaluation in reinforcement learning via reduction to stochastic variational inequality with Markovian noise \cite{kotsalis2022simple}.

The best-known result in the relative error setting \cite{gannot2022frequency,vasin2023accelerated} assumes that $\hat{\varepsilon} \lesssim \left(\mu/L\right)^{3/4}$ for the method to preserve the accelerated convergence rate. Here $L$ is the Lipschitz constant of the gradient and $\mu$ is the strong convexity parameter. In this paper, we improve this bound to $ \hat{\varepsilon} \lesssim \left(\mu/L\right)^{1/2}$, showing that accelerated gradient methods are more robust to relative error than it was known in the literature. Further, we propose two new families of intermediate methods parameterized by $p\in [1,2]$ and interpolating between non-accelerated and accelerated methods. The first family is based on the Similar Triangles Method from \cite{dvurechensky2018computational,d2021acceleration,gasnikov2018universal,gorbunov2019optimal} and the second one is based on \cite{devolder2013intermediate,dvurechensky2016stochastic,kamzolov2021universal}. The first family allows us to obtain the bound $\hat{\varepsilon} \lesssim \left(\mu/L\right)^{1/2}$ for $\mu$-strongly convex unconstrained problems and $ \hat{\varepsilon} \lesssim 1/N$, where $N$ is a number of required iterations, for convex unconstrained problems. Moreover, by using the PEP technique \cite{goujaud2022pepit,taylor2017smooth} we show that this result is the best possible for the considered family of methods. An interesting phenomenon we have observed with the first family is as follows: interpolation w.r.t. the acceleration level $p$ does not make sense for the robustness of algorithms in the family. All of the methods are equally robust to the level of noise. For the second family, we observed quite a different picture. Namely, we propose a proper adaptive way of choosing parameters $p$ and $L$ along the iteration process, which leads to better robustness by slowing down the convergence rate.



\bigskip
\noindent
\textbf{Paper organization}

The paper consists of an introduction and four main sections. The first of those four sections gives notation and necessary definitions. In Sect. \ref{sec:ISTM}  we
consider theoretical results for the Intermediate Similar Triangle Method for both convex and strongly convex optimization with relative noise in the gradient. Sect. \ref{sec:AIM} is devoted to adaptive algorithms for the constrained optimization problem. In particular, we investigate adaptivity w.r.t. both Lipschitz smoothness and intermediate parameter $p$.  In Sect. \ref{sec:exps} we present some numerical experiments that validate our theory and demonstrate the effectiveness of the proposed algorithms for the considered optimization problem. Finally, in Appendix (Sect. \ref{sec:appendix}) we provide missing proofs and additional experiments for our methods.

\section{Preliminaries}

We use $\langle x,y \rangle := \sum\limits_{k=1}^n x_ky_k$ to denote the inner product of $x,y \in \R^n$ and,   $\|x\|_2 := \left(\sum\limits_{k=1}^n x_k^2\right)^{1/2}$ to denote $\ell_2$-norm of $x \in \R^n$.

We will impose the following assumptions on the class of optimized functions $f$. 
\begin{assumption}[Convexity]\label{as:cvx}
    The objective function $f$ is convex, i.e.
    \begin{equation*}
    f(y) \geq f(x)+\langle\nabla f(x), y-x\rangle, \quad \forall x, y \in \mathbb{R}^n.
    \end{equation*}
\end{assumption}

\begin{assumption}[$L$-smootheness]\label{as:Lsmooth}
    There exist a constant $L > 0$ such that $f$ is $L$-smooth function, i.e.
\begin{equation}\label{smoothness_cond}
    f(y) \leq f(x)+ \left\langle\nabla f(x), y-x\right\rangle + \frac{L}{2} \|y-x\|_2^2, \quad \forall x, y \in \mathbb{R}^n,
\end{equation}
or equivalently 
\begin{equation}\label{eq_6}
    \|\nabla f(y) - \nabla f(x)\|_2 \leq L \|y - x\|_2, \quad \forall x, y \in \mathbb{R}^n.
\end{equation}
\end{assumption}

The convexity (Assumption  \ref{as:cvx}) and $L$-smoothness (Assumption \ref{as:Lsmooth}) can be combined with the next inequality \eqref{eq_8}.
\begin{proposition}
Let $f$ be a convex (Assumption \ref{as:cvx}) and $L$-smooth (Assumption \ref{as:Lsmooth}) function, then 
\begin{equation}\label{eq_8}
\left\|\nabla f(x)-\nabla f(y)\right\|_2^2 \leq 2 L \left(f(x)-f(y)-\langle\nabla f(y), x-y\rangle \right), \quad \forall x, y \in \mathbb{R}^n.
\end{equation}
\end{proposition}

We also consider a narrower class of strongly convex functions for the optimization problem \eqref{main_uncons_problem}.
 \begin{assumption}[Strong convexity]\label{as:str_cvx}
    There exists a constant $\mu > 0$, such that $f$ is $\mu$-strongly convex, i.e.
    \begin{equation}\label{eq:str_cvx}
        f(y) \geq f(x) + \langle \nabla f(x), y - x \rangle + \frac{\mu}{2}\|y - x\|_2^2, \quad x, y \in \R^n. 
    \end{equation}
\end{assumption}

Finally, we make an assumption about the noise in the accessible gradient of the objective function $f$.
\begin{assumption}[Relative noise]\label{as:noise}
We assume that we have access to a gradient of $f$ with relative noise, i.e.
\begin{equation}\label{eq_relative_error}
    \|\widetilde{\nabla} f(x) - \nabla f(x)\|_2 \leq \hat{\varepsilon} \|\nabla f(x)\|_2, \quad \forall \hat{\varepsilon} \in [0, 1].
\end{equation}
\end{assumption}

\section{Intermediate Similar Triangle Method with Relative Noise in Gradient}\label{sec:ISTM}
In this section, we consider the following unconstrained optimization problem 
\begin{equation}\label{main_uncons_problem}
    \min_{x \in \mathbb{R}^n} f(x).
\end{equation}

\subsection{Convex Case}
For the problem \eqref{main_uncons_problem}, in the considered setting above, we propose an algorithm, called Intermediate Similar Triangle Method ({\texttt{ISTM}}, see Algorithm \ref{alg_STM_relative}). This algorithm is a development of the ideas of the original Similar Triangle Method \cite{gasnikov2016universal} and intermediate acceleration \cite{devolder2013intermediate,dvurechensky2016stochastic,gorbunov2021near} with inexact oracle.

\begin{algorithm}[!ht]
\caption{ Intermediate  Similar Triangle Method (\texttt{ISTM}).}\label{alg_STM_relative}
\begin{algorithmic}[1]
   \REQUIRE Initial point $x^0$, number of iterations $N$, smoothness constant  $L>0$, and step size parameter $a \geq 1$, intermediate parameter $p \in [1,2]$. 
   \STATE Set $A_0 = \alpha_0 = 0, y^0 = z^0 = x^0$.
   \FOR{$k=0,1 ,  \ldots, N-1$}
   \STATE Set $\alpha_{k+1} = \frac{(k+2)^{p-1}}{2aL}, \, A_{k+1} = \alpha_{k+1} + A_k$. \label{item_3}
   \STATE $x^{k+1} = \frac{1}{A_{k+1}} \left(A_k y^k + \alpha_{k+1} z^k\right) $. \label{item_4}
   \STATE $z^{k+1} = z^k - \alpha_{k+1} \widetilde{\nabla} f(x^{k+1})$.
   \STATE $y^{k+1} = \frac{1}{A_{k+1}} \left(A_k y^k + \alpha_{k+1} z^{k+1}\right)$.
   \ENDFOR
    \ENSURE 
	$y^N$.
\end{algorithmic}
\end{algorithm}

The main results about the convergence of  \texttt{ISTM}  are presented in the following theorem.

\begin{theorem}\label{theo:coveregence_alg1}
Let function $f$ be convex (Assumption \ref{as:cvx}) and $L$-smooth (Assumption \ref{as:Lsmooth}) with relative noise $\hat{\varepsilon} \in [0,1]$ (Assumption \ref{as:noise}). Then after $N\geq 1$ iterations of  \texttt{ISTM} with intermediate parameter $p \in [1,2]$ and
\begin{equation}\label{formula_for_a}
    a = O\left(\max \left\{ 1, N^{\frac{p}{4}} \sqrt{\hat{\varepsilon}}, N^{\frac{p}{2}} \hat{\varepsilon}, N^p \hat{\varepsilon}^2 
 \right\}\right),
\end{equation}
we have 
\begin{equation}\label{conv_rate_alg1}
    f(y^N) - f(x^*) \leq \frac{8 a L C_1^2 R_0^2}{(N+1)^p},
\end{equation}
 where $R_0 = \|x^0 - x^*\|_2$ and 
\begin{equation}\label{eq_29}
     C_1 =  \sqrt{2}.
\end{equation}
i.e. with $a$ as in \eqref{formula_for_a}, we get 
\begin{equation}\label{eq:conv_rate_alg1_proper_a}
f(y^N) - f(x^*) \leq O\left(\max \left\{ \frac{LR_0^2}{N^p}, \frac{\sqrt{\hat{\varepsilon}} L R_0^2}{N^\frac{3p}{4}}, \frac{\hat{\varepsilon} LR_0^2}{N^\frac{p}{2}}, \hat{\varepsilon}^2 L R_0^2 \right\}\right).
\end{equation}
\end{theorem}

\noindent
\textbf{Sketch of the proof.}
From Lemma \ref{main_lemma2}, for all $N\geq 0$, we have 
\begin{equation}\label{eq_62_intro}
\begin{aligned}
A_N(f(y^N)-f(x^*)) \leq & \frac{1}{2}\left\|z^0-x^*\right\|_2^2-\frac{1}{2}\left\|z^N-x^*\right\|_2^2
\\& +\sum_{k=0}^{N-1} \alpha_{k+1}\left\langle\theta_{k+1}, x^*-z^k\right\rangle +\sum_{k=0}^{N-1} \alpha_{k+1}^2 \left\|\theta_{k+1} \right\|_2^2
\\&
 + \sum_{k=0}^{N-1} \alpha_{k+1}^2\left\langle\theta_{k+1}, \nabla f(x^{k+1})\right\rangle.
\end{aligned}
\end{equation}
Denoting $R_k = \|z^k - x^*\|_2$, we prove by induction that with a proper choice of the parameter $a$, we manage to get inequalities
\begin{equation}\label{eq_65_intro}
\begin{aligned}
 R_k^2 &\leq  R_0^2+2 \sum_{l=0}^{k-1} \alpha_{l+1}\left\langle\theta_{l+1}, x^*-z^l\right\rangle+2 \sum_{l=0}^{k-1} \alpha_{l+1}^2\left\langle\theta_{l+1}, \nabla f\left(x^{l+1}\right)\right\rangle \\& \;\;\;\; +2 \sum_{l=0}^{k-1} \alpha_{k+1}^2\left\|\theta_{l+1}\right\|_2^2 \leq C_1^2 R_0^2,
\end{aligned}
\end{equation}
hold for $k = 1, \ldots, N$ simultaneously, where $C_1 = \sqrt{2}$. 

By the induction, we bound every term of the sum in \eqref{eq_65_intro} using convexity, $L$-smoothness and relative noise (see Assumptions \ref{as:cvx}, \ref{as:Lsmooth}, and \ref{as:noise}). We also choose the parameter $a$ to be large enough in order to lower $\alpha_{l+1} = \frac{(l+2)^{p-1}}{2aL}$, and guarantee the bound. Namely, the final formula for the parameter $a$ looks like  
\begin{equation}\label{eq:a_proper_intro}
a = O\left(\max \left\{ 1, N^{\frac{p}{4}} \sqrt{\hat{\varepsilon}}, N^{\frac{p}{2}} \hat{\varepsilon}, N^p \hat{\varepsilon}^2 
\right\}\right),
\end{equation}
where each term in maximum corresponds to the term in the sum in \eqref{eq_65_intro}.

Therefore, putting bound \eqref{eq_65_intro} in \eqref{eq_62_intro} we obtain  
\begin{equation}\label{eq_66_intro}
     f(y^k) - f(x^*) \leq \frac{C_1^2 R_0^2}{2 A_k}, \quad k = 1, \ldots, N.
\end{equation}
Finally, we apply the inequality for coefficient $A_{k}\geq \frac{(k+1)^p}{4aL}$ from  Lemma \ref{main_lemma1}, and the formula for $a$ from \eqref{eq:a_proper_intro} to complete the proof. $\hfill\square$

\bigskip

\noindent
\textbf{{\tt{ISTM}} discussion}

In setup when relative noise $\hat{\varepsilon}$ is fixed, we are free to choose the number of iterations $N$ of \texttt{ISTM}. According to Theorem \ref{theo:coveregence_alg1}, the convergence rate is mostly affected by the proper choice of parameter $a$, i.e. 
\begin{eqnarray}
    a &=& O\left(\max \left\{ 1, N^{\frac{p}{4}} \sqrt{\hat{\varepsilon}}, N^{\frac{p}{2}} \hat{\varepsilon}, N^p \hat{\varepsilon}^2 
     \right\}\right), \label{eq: a_for_disccusion} \\
    f(y^N) - f(x^*) &\leq& O \left(\max \left\{ \frac{LR_0^2}{N^p}, \frac{\sqrt{\hat{\varepsilon}} L R_0^2}{N^{\frac{3p}{4}}}, \frac{\hat{\varepsilon} LR_0^2}{N^{\frac{p}{2}}}, \hat{\varepsilon}^2 L R_0^2
 \right\}\right).
\end{eqnarray}

As the number of iterations  $N$ increases, the proper parameter $a$ also increases in a step-wise manner and the convergence rate slows down. In other words, as long as the condition $N^p \hat{\varepsilon}^2 \leq 1$ is met, accuracy is increasing with every step of the \texttt{ISTM} with $\sim \frac{1}{N^p}$ rate.  However, as soon as this condition is violated, the accuracy reaches a plateau  $\hat{\varepsilon}^2 LR_0^2$. Moreover, the plateau value does not depend on the intermediate parameter $p$. In the Intermediate Gradient Method  \cite{dvurechensky2016stochastic} with absolute noise setup instead of relative one, the smaller intermediate parameter $p$, the more robust algorithm is to noise and the lower the accuracy that can be achieved. In the \texttt{ISTM} the same effect in theory is not observed. Therefore, the optimal choice in terms of the minimal number of necessary iterations $N$ or the fastest convergence rate is $p = 2$.  

\subsection{Strongly convex case}
For the strongly convex functions we use a restart technique to obtain a more robust towards relative noise algorithm called the Restarted Intermediate  Similar Triangle Method (\texttt{RISTM},  Algorithm \ref{alg_restartSTM_relative}). 

\begin{algorithm}[!ht]
\caption{ Restarted Intermediate  Similar Triangle Method (\texttt{RISTM}).} \label{alg_restartSTM_relative}
\begin{algorithmic}[1]
   \REQUIRE Initial point $x^0$, number of restarts $K$, numbers of iterations $\{N^i\}_{i=1}^K$, smoothness constants  $\{L^i\}_{i=1}^K$, step size parameters $\{a^i\}_{i=1}^K$ , intermediate parameter $p \in [1,2]$. 
   \FOR{$i= 1, \dots, K$}
   \STATE Start \texttt{ISTM} (Algorithm \ref{alg_STM_relative}) with parameters  $(x^i, N^i, L^i, a^i, p)$ and get output point $y^i$
   \STATE Set output point as new initial one: $x^{i+1} = y^i$. 
   \ENDFOR
    \ENSURE 
	$y^K$.
\end{algorithmic}
\end{algorithm}

In this case, it is possible to find the upper bound of noise $\hat{\varepsilon}$ when the convergence rate is the same as in the setup without noise. 

\begin{theorem}\label{theo:restart_conv}
Let $f$ be an $L$-smooth (Assumption \ref{as:Lsmooth}) and $\mu$-strongly convex (Assumption \ref{as:str_cvx}) function with relative noise $\hat{\varepsilon}$ (Assumption \ref{as:noise}). Also, let $\hat{\varepsilon}$ be small enough, i.e. 
$$\hat{\varepsilon} \lesssim \sqrt{\frac{ \mu}{4 L}}.$$ 
    
Then in order to achieve an accuracy $\varepsilon$, defined as $f(x) - f(x^*) \leq \varepsilon$, \texttt{RISTM} (Algorithm \ref{alg_restartSTM_relative}) with intermediate parameter $p \in [1,2]$ requires 
\begin{eqnarray}
K &=& \left\lceil\log_2\left(\frac{\mu R_0^2}{\varepsilon} \right) + 1\right\rceil \text{ restarts, } \\
N_{\text{total}} &=& \left\lceil \left( \frac{L}{\mu}\right)^{\frac1p} \log_2\left(\frac{\mu R_0^2}{\varepsilon} \right)\right\rceil \text{ total oracle calls,} 
\end{eqnarray}
with  $ N = \left\lceil\left( \frac{L}{\mu}\right)^{\frac1p}\right\rceil$ iterations and parameter $a$ chosen according to \eqref{formula_for_a} from Theorem \ref{theo:coveregence_alg1} at each restart.
 \end{theorem}
\begin{proof}
Let's take a look at first $N_1 =\left\lceil \left( \frac{L}{\mu}\right)^{\frac1p}\right\rceil$ iterations of \texttt{ISTM} with output point $x^1$. For the $\mu$-strongly  convex function $f$ and points $x^1, x^*$ due to \eqref{eq:str_cvx} the next inequality holds true
\begin{equation}\label{eq:strongly_convex_output}
   \frac{\mu}{2}\|x^1 - x^*\|_2^2 \leq f(x^1) - f(x^*). 
\end{equation} 
At the same time applying \texttt{ISTM} convergence Theorem \ref{theo:coveregence_alg1} with proper $a$ we bound right part of \eqref{eq:strongly_convex_output}  as \begin{equation}\label{eq:both_sides_eq}
\frac{\mu}{2}\|x^1 - x^*\|_2^2 \leq f(x^1) - f(x^*) \leq O \left(\max \left\{ \frac{LR_0^2}{N^p}, \frac{\sqrt{\hat{\varepsilon}} L R_0^2}{N^{3p/4}}, \frac{\hat{\varepsilon} LR_0^2}{N^{p/2}}, \hat{\varepsilon}^2 L R_0^2
 \right\}\right).
\end{equation}
Since $\hat{\varepsilon} \lesssim \sqrt{\frac{\mu}{4 L}} $, and $N_1 = \left \lceil\left( \frac{L}{\mu}\right)^{\frac1p} \right\rceil$, putting them into \eqref{eq:both_sides_eq} we get
\begin{eqnarray}\label{eq:bound_acc}
\frac{\mu}{2}\|x^1 - x^*\|_2^2 &\leq& f(x^1) - f(x^*) \notag \\
&\leq& O \left(\max \left\{ \frac{\mu R_0^2}{4}, \frac{\mu R_0^2}{4}, \frac{\mu R_0^2}{4}, \frac{\mu R_0^2}{4}
 \right\}\right) \leq \frac{\mu R_0^2}{4}.
\end{eqnarray}
Denoting $R_k = \|x^k - x^*\|_2$, we find that $R_0^2$ halves after $N_1$ iterations 
\begin{eqnarray}
     R_1^2 \leq \frac{R_0^2}{2} \label{eq:halves}.
\end{eqnarray}
On the next restart, we take $x^1$ as initial point $x^0$ and distance to solution $R_1$ as initial distance $R^0$. Thus, after $K$ restarts, we get the next bound for distance to the solution from \eqref{eq:halves}
\begin{eqnarray}
    R_K^2 \leq \frac{R_0^2}{2^K} \notag, 
\end{eqnarray}
and bound for accuracy from \eqref{eq:bound_acc}
\begin{eqnarray*}
    f(x^K) - f(x^*) \leq \frac{\mu R_{K-1}^2}{4} \leq \frac{\mu R_0^2}{4 \cdot 2^{K-1}} .
\end{eqnarray*}
Considering the total number of restarts  $K =\left \lceil \log_2\left(\frac{\mu R_0^2}{\varepsilon} \right) + 1 \right\rceil$ we, finally, have
\begin{eqnarray*}
    f(x^K) - f(x^*) \leq \frac{\varepsilon}{4}.
\end{eqnarray*}
All that remains is to notice that in each restart we have $ \left\lceil \left( \frac{L}{\mu}\right)^{\frac1p}\right\rceil$ iterations and total number of oracle calls equals $\left\lceil \left( \frac{L}{\mu}\right)^{\frac1p}\right\rceil\cdot K.$ 
\end{proof}

\bigskip

\noindent
\textbf{{\tt{RISTM}} discussion}

Theorem \ref{theo:restart_conv} states that the upper bound of relative noise that does not affect the convergence rate of \texttt{RISTM} is $\hat{\varepsilon} \lesssim \sqrt{\frac{\mu}{4 L}} $. In the best case of $p=2$, the Algorithm converges linearly with coefficient $\sqrt{\frac{L}{\mu}}$ that corresponds to optimal first-order methods, like for example Nesterov Accelerated Gradient. In the previous work \cite{vasin2023accelerated}, only $\hat{\varepsilon} \lesssim \left(\frac{\mu}{L}\right)^\frac34$ upper bound was obtained and authors made a hypothesis of improvement up to $\sqrt{\frac{\mu}{ L}}$ that we proved.

\section{Adaptive Intermediate Method with Relative Noise in Gradient}\label{sec:AIM}

Let $Q \subseteq \mathbb{R}^n$ be a simple closed convex set and $f: Q \longrightarrow \mathbb{R}$ be a convex smooth function.  We consider the following convex optimization problem
\begin{equation}\label{cons_problem}
    \min_{x \in Q} f(x).
\end{equation}





For problem \eqref{cons_problem}, we consider an adaptive algorithm called Adaptive Intermediate Algorithm (\texttt{(AIM)}, see Algorithm \ref{adaptive_alg}). The adaptation in this algorithm is for the constant $L$. 

\begin{algorithm}[htp]
\caption{Adaptive Intermediate Method (\texttt{AIM}).}\label{adaptive_alg}
\begin{algorithmic}[1]
   \REQUIRE $L_s>0$ initial guess for the Lipschitz constant, intermediate parameter $p \in [1,2]$, positive sequence $\{\delta_k\}_{k \geq 0}$.  
   \STATE set $z^0 = x^0$.
   \STATE Set $i_0 = 0$.
   \STATE Compute
   \begin{equation}\label{y0_adaptive_alg}
       y_0 = \arg\min_{x \in Q} \left\{ \frac{1}{2^{i_0} L_s} \left\langle \widetilde{\nabla} f(x^0), x - x^0 \right\rangle + \frac{1}{2}\|x - x^0\|_2^2\right\}.
   \end{equation}
   \STATE \textbf{If}
   \begin{equation*}
       f(y^0) \leq f(x^0) + \left\langle \widetilde{\nabla} f(x^0), y^0 - x^0 \right\rangle + \frac{2^{i_0} L_s}{2} \left\|y^0 - x^0\right\|_2^2 + \delta_0,
   \end{equation*}
   \textbf{then} go to step 5. Otherwise $i_0 = i_0 + 1$ and go to step 3.
   \STATE Define $L_0 = 2^{i_0} L_s, \alpha_0 = B_0 = A_0 = \frac{1}{L_0}$.
   \FOR{$k=1 , 2,   \ldots$}
   \STATE Set $i_k = 0$.
   \STATE Set $L_k = 2^{i_k} L_{k-1}$ and 
   \begin{equation}\label{alpha_k_adaptive_alg}
       \alpha_{k} = \frac{1}{L_k} \left(\frac{k + 2p}{2p}\right)^{p-1},
   \end{equation}
   \begin{equation}\label{B_k_adaptive_alg}
       B_k = \alpha_k^2 L_k,
   \end{equation}
   \begin{equation}\label{xk_adaptive_alg}
   x^k = \frac{\alpha_k}{B_k} z^{k-1} + \left(1-\frac{\alpha_k}{B_k}\right) y^{k-1},
   \end{equation}
    \begin{equation}\label{zk_adaptive_alg}
   z^k = \arg\min_{x \in Q} \left\{ \frac{1}{2}\|x - x^0\|_2^2 + \sum_{i = 0}^{k} \alpha_i  \left\langle \widetilde{\nabla} f(x^i), x - x^i \right\rangle \right\},
   \end{equation}
   \begin{equation}\label{wk_adaptive_alg}
   w^k = \frac{\alpha_k}{B_k} z^{k} + \left(1-\frac{\alpha_k}{B_k}\right) y^{k-1},
   \end{equation}
   \STATE  \textbf{If}
   \begin{equation}\label{criter_out_iter}
       f(w^k) \leq f(x^k) + \left\langle \widetilde{\nabla} f(x^k), w^k - x^k \right\rangle + \frac{ L_k}{2} \|w^k - x^k\|_2^2 + \delta_k,
   \end{equation}
   \textbf{then} go to step 10. Otherwise $i_k = i_k + 1$ and go to step 8.
   \STATE
   \begin{equation}\label{Ak_adaptive_alg}
       A_k = A_{k-1} + \alpha_k,
   \end{equation}
   \begin{equation}\label{yk_adaptive_alg}
       y^k = \frac{B_k}{A_k} w^{k} + \left(1-\frac{B_k}{A_k}\right) y^{k-1}. 
   \end{equation}
   \ENDFOR
    \ENSURE 
	$y^k$.
\end{algorithmic}
\end{algorithm}


Let us mention the following lemma which gives an upper bound for $A_k f(y^k)$. The proof will be by induction on $k \geq 0$ and similar to the proof of Theorem 3.1 in  \cite{kamzolov2021universal} (For the full proof of this lemma, see Subsect. \ref{sect_missing_proofs_AIM}).

\begin{lemma}\label{lemma_upper_bound_Ak_fk}
Let $f$ be convex (Assumption \ref{as:cvx}) and smooth (Assumption \ref{as:Lsmooth}) function. By Algorithm \ref{adaptive_alg}, we have 
\begin{equation}\label{ineq_upper_bound_Ak_fk}
    A_k f(y^k) - M_k \leq \Psi_k^*, \quad k \geq 0,
\end{equation}
where 
$M_k = \sum_{i = 0}^{k} B_i \delta_i$,
and
\begin{equation}\label{psi_star}
\Psi_k^*=\min _{x \in Q}\left\{\Psi_k(x) : = \frac{1}{2}\|x - x^0\|_2^2 + \sum_{j=0}^k \alpha_j\left(f (x^j) + \left\langle \widetilde{\nabla} f(x^j), x - x^j \right\rangle\right) \right\}.
\end{equation} 
\end{lemma}

From Lemma \ref{lemma_upper_bound_Ak_fk}, we can conclude the following corollary.

\begin{corollary}\label{corr_rate_adaptive}
Let $f$ be a convex (Assumption \ref{as:cvx}) and smooth (Assumption \ref{as:Lsmooth}) function.  For all $k \geq 0$, we have 
\begin{equation}\label{adaptive_estimate}
    f(y^k) - f(x^*) \leq \frac{h(x^*)}{A_k} + \frac{1}{A_k} \sum_{i = 0}^{k}B_i \delta_i, 
\end{equation}
where $h(x^*) := \frac{1}{2} \|x^* - x^0\|_2^2$. 
\end{corollary}

For the convergence rate of \texttt{AIM} (Algorithm \ref{adaptive_alg}), we establish it in the following theorem. 

\begin{theorem}\label{the:rate_adaptive_alg}
Let $f$ be convex (Assumption \ref{as:cvx}) and smooth (Assumption \ref{as:Lsmooth}) function. Let  $p \in [1,2]$ be the intermediate parameter, and let us assume that an upper bound $R_0$ for the distance from the starting point $x_0$ to a solution $x^*$  is known. Then, with sequences $\{\alpha_k\}_{k \geq 0}, \{B_k\}_{k \geq 0}$ given in \eqref{alpha_k_adaptive_alg}, \eqref{B_k_adaptive_alg}, respectively, and sequence $\{\delta_k\}_{k \geq 0}$, for \texttt{AIM}, we have
\begin{equation}\label{estimate2_adptive_alg}
    f(y^k) - f(x^*) \leq \frac{ 4  R_0^2 }{(k+2)^p} \left(\max_{0 \leq i \leq k} L_i\right)  + 2\left(\max_{0 \leq i \leq k} \delta_i\right) k^{p-1}, \quad \forall k \geq 1.
\end{equation}
\end{theorem}
\textbf{Sketch of the proof.} For any $p \in [1, 2]$,  we have 
$$
 A_k \geq \frac{1}{L} \left(\frac{k+2}{4}\right)^p, \quad \sum_{j = 0}^{k} B_j \delta_j \leq  \left(\frac{k + 2p} {2p}\right)^{p-1}  \left(\max_{0 \leq i \leq k} \delta_i\right) A_k.
$$
Therefore from Corollary \ref{corr_rate_adaptive}, we get 
$$
 f(y^k) - f(x^*) \leq \frac{ 4  R_0^2  }{(k+2)^p}\left(\max_{0 \leq i \leq k} L_i\right) + 2\left(\max_{0 \leq i \leq k} \delta_i\right) k^{p-1},
$$
where $h(x^*) = \frac{1}{2}\|x^* - x^0\|_2^2 \leq R_0^2$. For the full proof, see Subsect. \ref{sect_missing_proofs_AIM}.

\bigskip 

Now, from Assumptions \ref{as:Lsmooth} and \ref{as:noise}, for any $ x, y \in Q$, we have 
\begin{equation*}
    \begin{aligned}
        f(y) & \leq f(x) + \left\langle \widetilde{\nabla}  f(x), y - x \right\rangle + \frac{L}{2} \left\|y - x\right\|_2^2  + \left\langle \nabla f(x) - \widetilde{\nabla} f(x), y - x \right\rangle
        \\& \leq f(x) + \left\langle \widetilde{\nabla}  f(x), y - x \right\rangle + \frac{L}{2} \left\|y - x\right\|_2^2  + \left\| \widetilde{\nabla} f(x) - \nabla f(x)\right\|_2 \cdot \| y  - x\|_2
        \\& \leq f(x) + \left\langle \widetilde{\nabla}  f(x), y - x \right\rangle + \frac{L}{2} \left\|y - x\right\|_2^2  + \hat{\varepsilon}\left\| \nabla f(x)\right\|_2 \cdot \| y  - x\|_2.
    \end{aligned}
\end{equation*}
But, from Assumption \eqref{as:noise}, we get the following
\begin{equation}\label{res_relative_noise}
    (1-\hat{\varepsilon}) \left\|\nabla f(x)\right\|_2 \leq \|\widetilde{\nabla} f(x)\|_2  \leq (1+\hat{\varepsilon}) \left\|\nabla f(x)\right\|_2, \quad \forall x \in Q.
\end{equation}
From this, we have 
\begin{equation}
    \left\|\nabla f(x)\right\|_2 \leq \frac{1}{1-\hat{\varepsilon}} \|\widetilde{\nabla} f(x)\|_2, \quad \forall x \in Q.
\end{equation}
Therefore, for any $x , y \in Q$,  we have 
\begin{equation*}
    \begin{aligned}
        f(y) & \leq f(x) + \left\langle \widetilde{\nabla}  f(x), y - x \right\rangle + \frac{L}{2} \left\|y - x\right\|_2^2  + \frac{\hat{\varepsilon}}{1-\hat{\varepsilon}} \|\widetilde{\nabla} f(x)\|_2 \cdot \| y  - x\|_2
        \\& \leq f(x) + \left\langle \widetilde{\nabla}  f(x), y - x \right\rangle + \frac{L}{2} \left\|y - x\right\|_2^2  + \frac{c L}{2} \| y  - x\|_2^2 
        \\& \;\;\;\; +  \frac{\hat{\varepsilon}^2}{c(1-\hat{\varepsilon})^2} \|\widetilde{\nabla} f(x)\|_2^2,
    \end{aligned}
\end{equation*}
for any $c >0$. This means that, for any $x, y \in Q$, we have
\begin{equation}\label{eq_res_smooth_relative}
    f(y)  \leq     f(x) + \left\langle \widetilde{\nabla}  f(x), y - x \right\rangle + \frac{ (1 + c) L}{2} \| y  - x\|_2^2   +  \frac{\hat{\varepsilon}^2}{c(1-\hat{\varepsilon})^2} \|\widetilde{\nabla} f(x)\|_2^2.
\end{equation}
Therefore, for some fixed $\hat{c}$, such that: $ \hat{c} \leq c (1-\hat{\varepsilon})^2; \, \forall c >0$, and with $\hat{L} : = L\left(1+ \frac{\hat{c}}{(1-\hat{\varepsilon})^2}\right) $, for any $x, y \in Q$, it holds
\begin{equation}\label{eq_fdf4112}
    f(y) \leq f(x) + \left\langle \widetilde{\nabla}  f(x), y - x \right\rangle + \frac{\hat{L}}{2} \left\|y - x\right\|_2^2 + \frac{\hat{\varepsilon}^2\|\widetilde{\nabla} f(x)\|_2^2 }{\hat{c}} .
\end{equation}

From \eqref{eq_fdf4112}, we find that for each $\delta_k>0$ (we can put $\delta_k = \frac{\hat{\varepsilon}^2\|\widetilde{\nabla} f(x^k)\|_2^2 }{\hat{c}}$ for some $\hat{c} >0$) there is such $L_k > 0$ that the following inequality holds
$$
f(x^{k+1}) \leq f(x^k) + \left\langle \widetilde{\nabla}  f(x^k), x^{k+1} - x^k \right\rangle + \frac{L_k}{2} \left\|x^{k+1} - x^k\right\|_2^2 + \delta_k.
$$
Therefore, we can apply Algorithm \ref{adaptive_alg} to the problem under consideration \eqref{cons_problem}, with access to relative noise in the gradient $\widetilde{\nabla} f(x)$ of the objective function $f$, and due to the adaptivity we can take $L_k < \hat{L}$ and decrease $\delta_k$.

\begin{remark}
Depending on \eqref{estimate2_adptive_alg}, we can reconstruct \texttt{AIM}, and propose an adaptive algorithm, with adaptation in two parameters $L$ and $p$ (see Algorithm \ref{adaptive_alg_L_p}). 
In this algorithm, we start from the case when $p = 2$, which is the case for which the estimate \eqref{estimate2_adptive_alg} is the best at the first iterations. Then at the iteration, when the estimate \eqref{estimate2_adptive_alg} deteriorates we decrease $p$ until $p=1$. As a result, by Algorithm \ref{adaptive_alg_L_p} we get a solution to the problem  \eqref{cons_problem}, with an estimate of a solution better than an estimate with fixed $ p \in [1, 2]$ as in \texttt{AIM}. See Figs. \ref{res_alg1_alg3_a_5_15} and \ref{res_adaptive2_error_09_099} (in the right).   
\end{remark}

\begin{algorithm}[htp]
\caption{Adaptive Intermediate Method in $L$ and $p$ (\texttt{AIM} with variable $p$).}\label{adaptive_alg_L_p}
\begin{algorithmic}[1]
   \REQUIRE $x^0$ initial point, $L_0 > 0$ initial guess for the Lipschitz constant, positive sequence $\{\delta_k\}_{k \geq 0}$, $R_0$ s.t. $R_0^2 \geq \frac{1}{2} \left\|x^*- x^0\right\|_2^2 $, $\eta \ll 1$ small enough parameter.
   \STATE Set $p_0:= 2$, and  $E_0 = L_0 R_0^2$.
   \FOR{$k=0, 1,   \ldots$}
   \STATE Run \texttt{AIM} (Algorithm \ref{adaptive_alg}) with parameters $x^k, L_k, p_k$, i.e. run \texttt{AIM}$(x^k, L_k, p_k)$.
   \STATE
   Calculate
   $$
   E_{k+1} : = \frac{ 4  R_0^2 }{(k+3)^{p_k}} \left(\max_{0 \leq i \leq k+1} L_i \right) + 2\left(\max\limits_{0 \leq i \leq k+1} \delta_i\right) (k+1)^{p_k-1}.
   $$
   \STATE  \textbf{If} $E_{k+1} \leq E_{k}$, \textbf{then} go to the next iteration $k \to k+1$. \textbf{Else} run \texttt{AIM}$(x^k, L_k, p_k - \eta)$. 
   \ENDFOR
\end{algorithmic}
\end{algorithm}

\begin{remark}
If we assume that for any $x \in Q$, where $Q$ is bounded, we have $\|\widetilde{\nabla} f(x)\|_2 \leq M$ for some $M>0$. Then from \eqref{estimate2_adptive_alg}, when $p = 1$, we get 
\begin{equation*}
    \begin{aligned}
        f(y^k) - f(x^*) & \leq \frac{4R_0^2}{k + 2} \left(\max_{0 \leq i \leq k} L_i\right) + 2 \left(\max_{0 \leq i \leq k } \delta_i\right)
        \leq \frac{8\hat{L}R_0^2}{k+2} + \frac{\hat{\varepsilon}^2 M^2 }{\hat{c}}.
    \end{aligned}
\end{equation*}
From this we find 
$$
f(y^k ) - f(x^*) \leq O\left(\max \left\{ \frac{\hat{L}R_0^2}{k+2}, \frac{\hat{\varepsilon}^2 M^2 }{\hat{c}} \right\}\right), \quad \forall k \geq 0.
$$
Therefore, we got a similar result as for Algorithm \ref{alg_STM_relative} with $p = 1$. 
\end{remark}


\section{Numerical Experiments}\label{sec:exps}

The code we used in experiments is available at \\
https://github.com/Jhomanik/InterRel.

\subsection{\texttt{ISTM} PEP Experiments}
In order to find the tight convergence rate that \texttt{ISTM} (Algorithm \ref{alg_STM_relative}) achieves in Theorem \ref{theo:coveregence_alg1}, we consider next optimization problem with fixed $N, \hat{\varepsilon}$, and $a$
\begin{equation}\label{eq:max_problem_intro}
 \tau^N := \max\limits_{n, f, x^0} f(x^N) - f(x^*),   
\end{equation}
subject to: $f:\R^n \rightarrow \R$ is an convex and $L$-smooth  function, $0 \in \nabla f(x^*)$,  and
\begin{eqnarray} 
    && \|x^0 - x^*\|^2_2 \leq R^2,  \\
    && \|\widetilde{\nabla} f(x^k) - \nabla f(x^k)\|^2_2 \leq \hat{\varepsilon}^2 \|\nabla f(x^k)\|^2_2, \quad k = \overline{0,  N-1}, \label{eq:pep_rel} \\
    && x^{k+1}, y^{k+1}, z^{k+1} = \text{ISTMstep}(x^{k}, y^k, z^k, \widetilde{\nabla} f(x^k)), \quad k = \overline{0,  N-1}.\label{eq:pep_istm}
\end{eqnarray}
In other words, we are looking for the function $f$ from the set of convex and $L$-smooth functions that after $N$ iterations of \texttt{ISTM} \eqref{eq:pep_istm} with relative noise $\hat{\varepsilon}$ \eqref{eq:pep_rel}, gives the worst possible accuracy $\max\limits_{f} f(x^N) - f(x^*)$. 

This problem is usually called the Performance Estimation Problem (PEP) and can be equivalently reformulated as a semi-definite program (SDP), see \cite{de2017worst,drori2014performance,taylor2017smooth}.  Then we numerically solve this problem via MOSEK solver \cite{aps2019mosek}. We provide additional details about PEP in Section \ref{sec:pep_theory}.

\bigskip
\noindent
\textbf{Experiment 1.}

We check how optimal our choice of the parameter $a$ of \texttt{ISTM} is for the given relative noise $\hat{\varepsilon}$, intermediate parameter $p$, and number of iterations $N$. We will work under a $L$-smooth and convex setting. Design is as follows: for all $i$ from $0$ to $N_{\max}$ with fixed $a, \hat{\varepsilon}, p$, we calculate the value of the maximum overall convex, $L$ smooth functions difference $ \tau_i = \max_f f(y^i) - f(x^*)$ from the maximization problem \eqref{eq:max_problem} using the PEPit framework. Then we look for the moment $N_{\text{pep}}(a, \hat{\varepsilon}, p)$ at which the sequence of $\{\tau_i\}_{i=0}^{N_{\max}}$ stops to decrease. We compare numerically calculated map $N_{\text{pep}} (a, \hat{\varepsilon}, p)$ with theoretical map $a = C^2 N^p \hat{\varepsilon} ^2$ or $N_{\text{theory}}(a, \hat{\varepsilon}, p) = \left(\frac{C^2 a}{\hat{\varepsilon}^2}\right)^{\frac{1}{p}}$.

Fig. \ref{fig:conv_with_opt_a_diff_p} shows the dependence of the maximum difference $\tau^N$ on the number of iterations $N$ for different $p$ and  2 levels of relative noise $\hat{\varepsilon} = 0.75; 0.95$ with smoothness constant $L = 1$ and initial radius $R = 1$.

\begin{figure}[htp]
\centering
{\includegraphics[width=5.5cm]{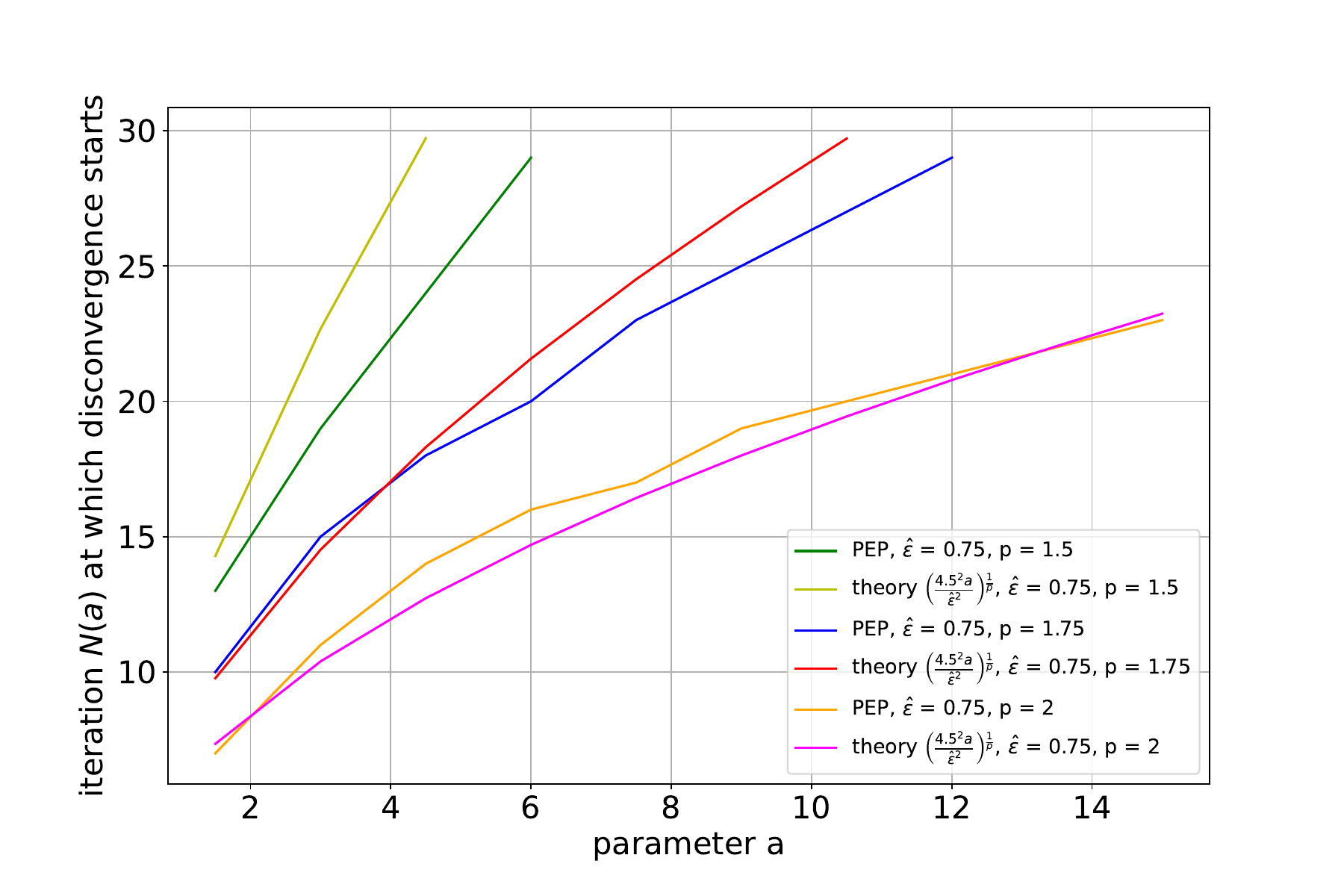} }
{\includegraphics[width=5.5cm]{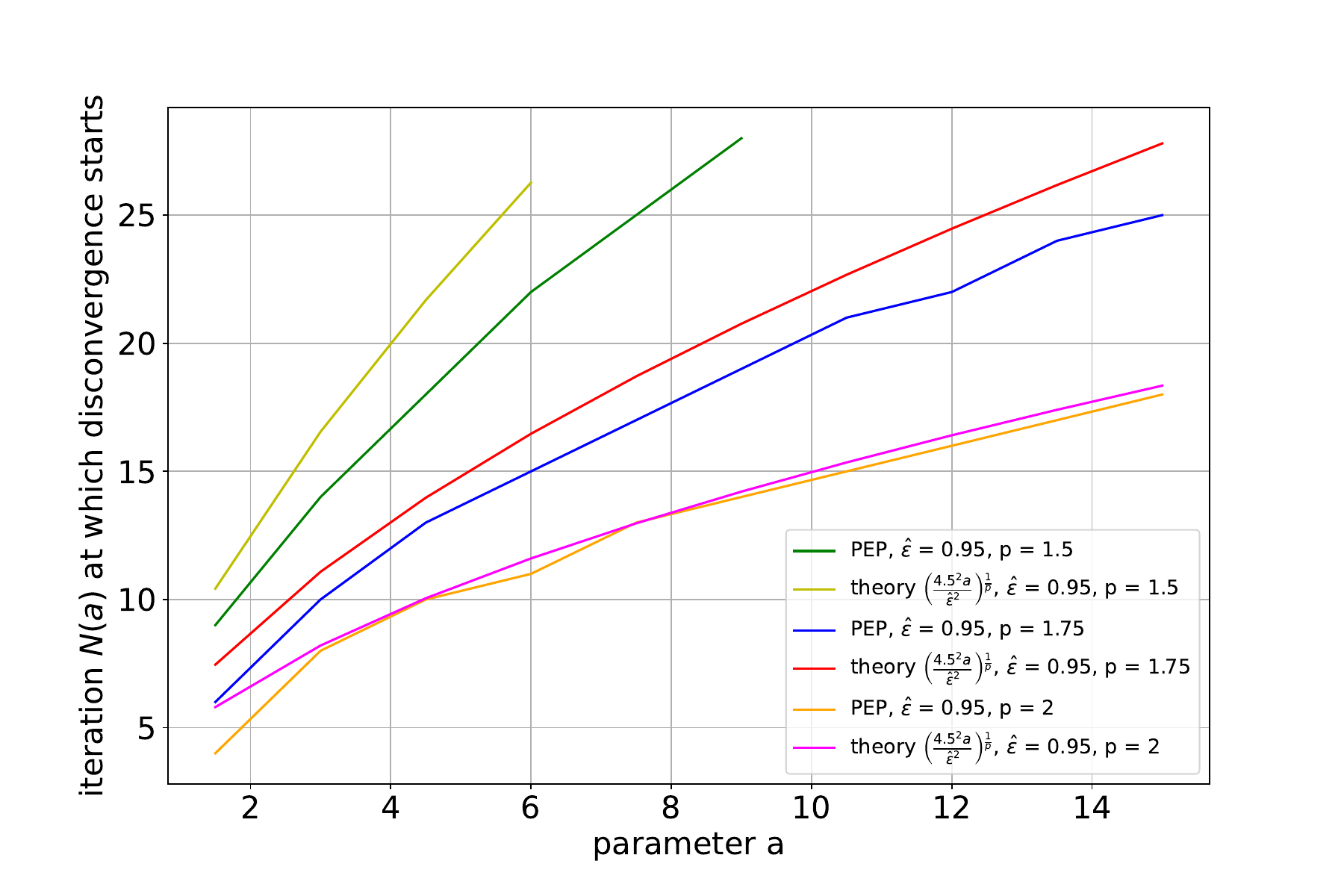} }
\caption{Comparison of theoretical and PEP calculated iteration $N$ after which \texttt{ISTM} with parameter $a$ starts to disconverge for different $p$ under $ L = 1, R = 1$ and \textbf{Left:} $\hat{\varepsilon} = 0.75$, \textbf{Right:} $\hat{\varepsilon} = 0.75$.  }
\label{fig:conv_with_opt_a_diff_p}
\end{figure}
As one can see theoretical and numerical results in the case of arbitrary $p \in [1,2]$ matches well, especially in terms of the forms of the curves. Constant $C$ at which the smallest error of approximation is achieved is $C \sim 4.5$.

\bigskip
\noindent
\textbf{Experiment 2.}

In this experiment, we study the effect of intermediate acceleration on plateau values that \texttt{ISTM} achieves. By convergence formula \eqref{eq:conv_rate_alg1_proper_a} from Theorem \ref{theo:coveregence_alg1},
$$f(y^N) - f(x^*) \leq O\left(\max \left\{ \frac{LR_0^2}{N^p}, \frac{\sqrt{\hat{\varepsilon}} L R_0^2}{N^\frac{3p}{4}}, \frac{\hat{\varepsilon} LR_0^2}{N^\frac{p}{2}}, \hat{\varepsilon}^2 L R_0^2 ,
 \right\}\right)$$
 the plateau level at which the algorithm begins to oscillate after a large number of iterations $N$ does not depend on $p$ and equals $LR_0^2\hat{\varepsilon}^2$. To do this, we consider solving maximization problem \eqref{eq:max_problem} $ \tau^N = \max_f f(y^N) - f(x^*)$ setting the parameter $a$ of \texttt{ISTM}  to $ \max \left\{ 1, N^{\frac{p}{4}} \sqrt{\hat{\varepsilon}}, N^{\frac{p}{2}} \hat{\varepsilon}, N^p \hat{\varepsilon}^2 \right\}$. 
 
 Fig. \ref{fig:plateau_diff_p} shows the convergence of \texttt{ISTM} with proper parameter $a$ for different intermediate parameter $p$. 
 
\begin{figure}[htp]
\centering
{\includegraphics[width=7.5cm]{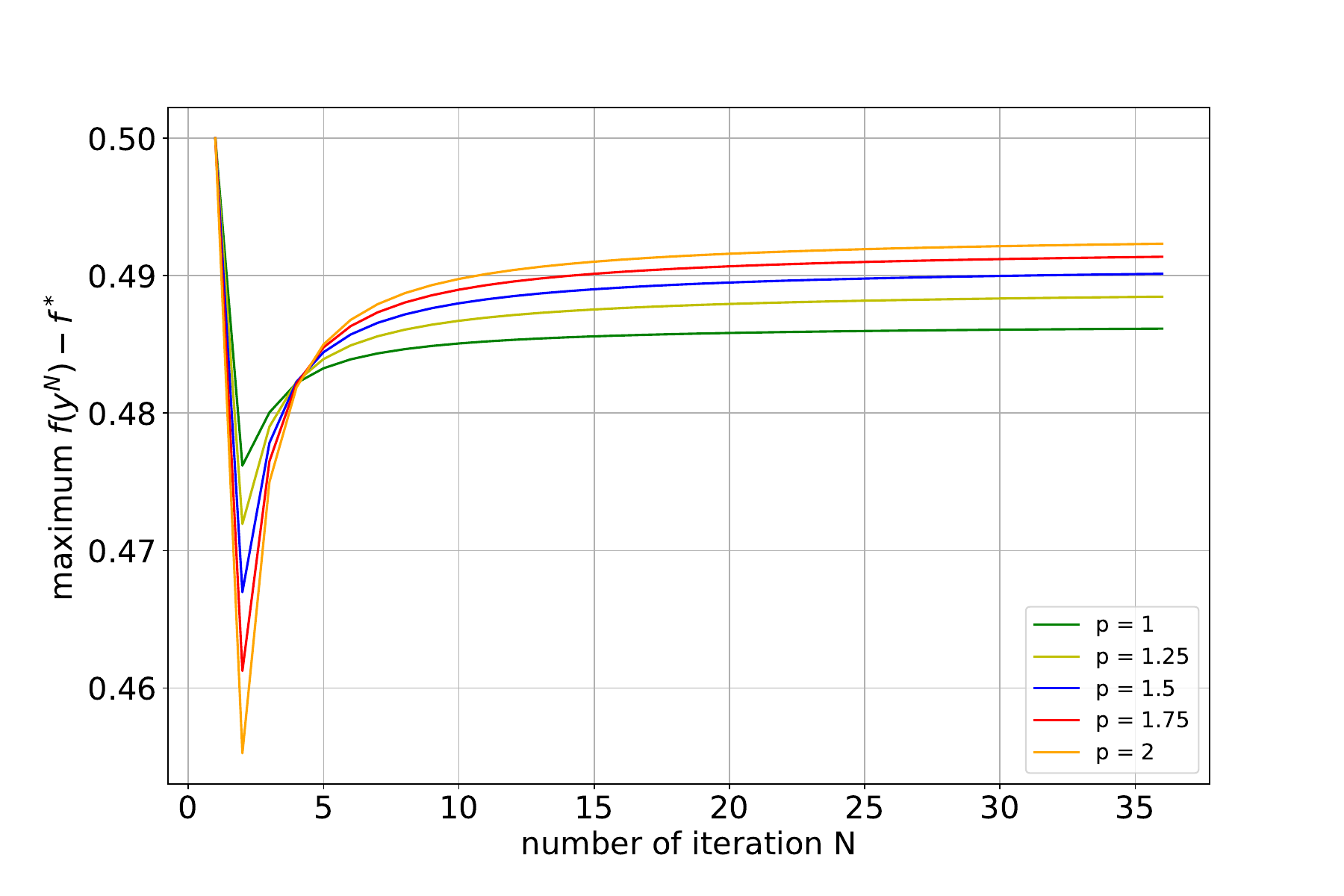} }
\caption{ Plateau level that achieve \texttt{ISTM} with different intermediate parameter $p$ and $L = 1, R = 1, \hat{\varepsilon} = 0.95$.}
\label{fig:plateau_diff_p}
\end{figure}
One can observe that there are differences in the rate of convergence to a plateau, but the plateau levels themselves differ very little from each other. Therefore, PEP confirms the results of Theorem \ref{theo:coveregence_alg1}.

\bigskip
\noindent
\textbf{PEP Experiments discussion}

Based on all the experiments performed, we can conclude that our Theorem \ref{theo:coveregence_alg1} gives tight estimates for the convergence of \texttt{ISTM}. We conjecture that up to the numerical constant, it is impossible to improve them in terms of robustness to relative noise and convergence rate.

\subsection{Comparison of Algorithms \ref{alg_STM_relative}, \ref{adaptive_alg} and \ref{adaptive_alg_L_p}. }  \label{subsec_major_experim}

In this subsection, to demonstrate the performance of the proposed Algorithms \ref{alg_STM_relative} (\texttt{ISTM}), \ref{adaptive_alg} (\texttt{AIM}) and \ref{adaptive_alg_L_p} (\texttt{AIM} with variable $p$), we conduct some numerical experiments for the considered problem \eqref{main_uncons_problem}, with the following objective function, 
methods (see \cite{nesterov2018lectures}). 
\begin{equation}\label{ex_smooth}
    f(x) = \frac{L}{8} \left(x_1^2 + \sum_{i=1}^{n-1} (x_i - x_{i+1})^2 + x_n^2\right) - \frac{L}{4} x_1,
\end{equation}
for some $L>0$. This function is known as the worst-case function for first-order. It is an $L$-smooth, and for it, we have
$$
x^* = \left(1-\frac{1}{n+1}, 1- \frac{2}{n+1}, \ldots, 1- \frac{n}{n+1}\right) \in \mathbb{R}^n. 
$$

We run Algorithms \ref{alg_STM_relative}, \ref{adaptive_alg} and \ref{adaptive_alg_L_p} with $n = 100, L = 1, p = 2$ with different values of $\hat{\varepsilon} \in [0, 1]$. In Algorithm \ref{alg_STM_relative} we take $a = 2$ and in Algorithms \ref{adaptive_alg} and \ref{adaptive_alg_L_p} we take $\delta_k =\frac{\hat{\varepsilon}^2 \|\widetilde{\nabla} f(x^k)\|_2^2}{\hat{c}}$, with $\hat{c} = 1000$. 

The results of the work of these compared algorithms are represented in Fig. \ref{res_alg1_alg3_a_5_15}, below. These results demonstrate the difference $f(y^k) - f(x^*)$ at each iteration and theoretical estimates 
\eqref{eq_66_intro} (for \texttt{ISTM}), and estimates \eqref{adaptive_estimate} (which is labeled in the figure by "AIM (estimate 1)"),  \eqref{estimate2_adptive_alg} (which is labeled in the figure by "AIM (estimate 2)")  for \texttt{AIM}.

\begin{figure}[htp]
\centering
{\includegraphics[width=5.5cm]{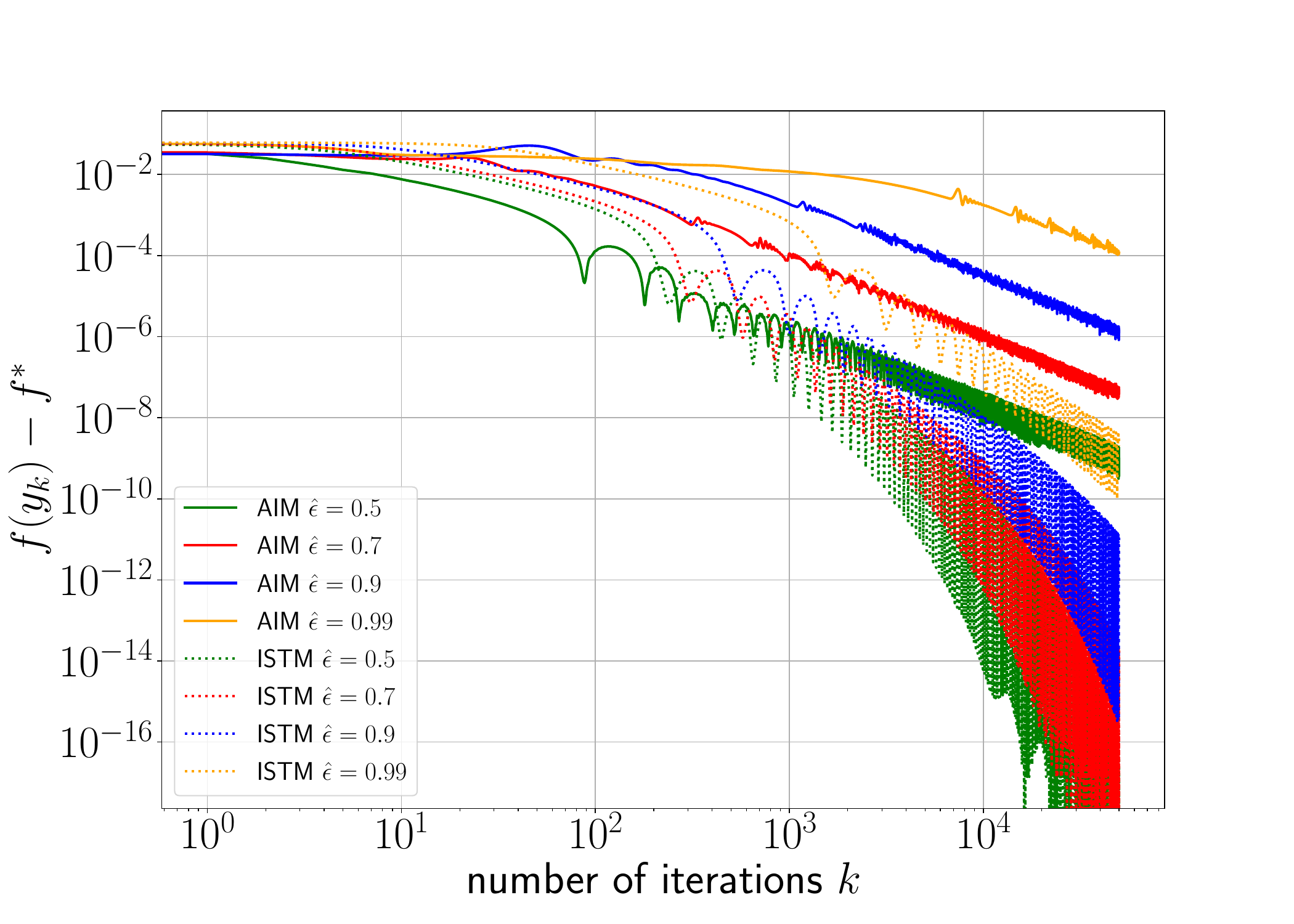} }
{\includegraphics[width=5.5cm]{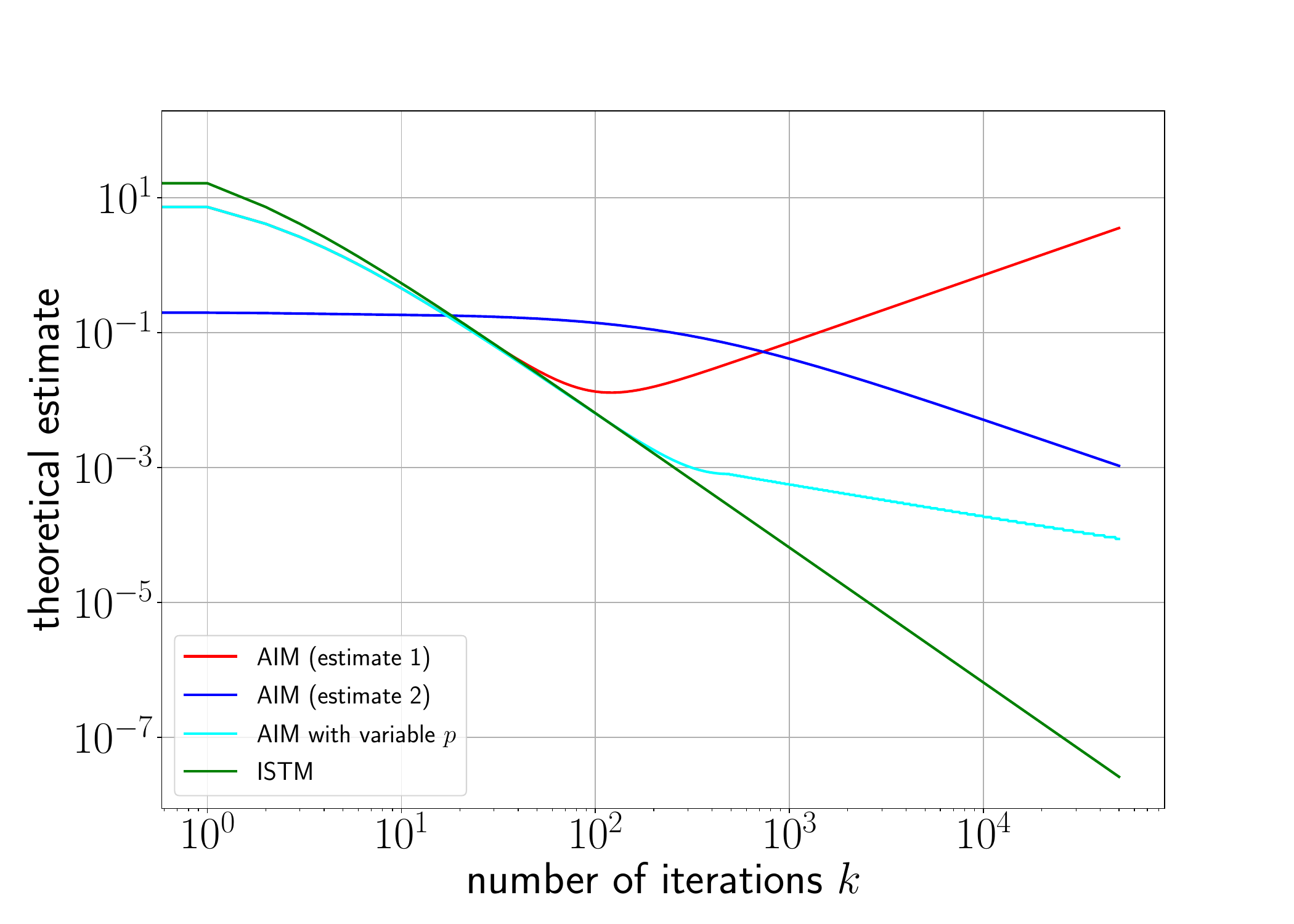} }
\caption{The results of \texttt{ISTM} (Algorithm \ref{alg_STM_relative}), \texttt{AIM} (Algorithm \ref{adaptive_alg}) and \texttt{AIM} with variable $p$ (Algorithm \ref{adaptive_alg_L_p}) for  the objective function \eqref{ex_smooth} with different values of $\hat{\varepsilon} \in [0,1]$. }
\label{res_alg1_alg3_a_5_15}
\end{figure}

From Fig.  \ref{res_alg1_alg3_a_5_15}, for the objective function \eqref{ex_smooth}, we see that through \texttt{ISTM}, we can, for not big values of the parameter $a$, obtain better minimal values for the objective function than the values we obtain through \texttt{AIM} (when we increase the value of $a$, at the first iterations of \texttt{AIM} will give better minimal values of the objective than \texttt{ISTM}, see also Fig. \ref{res_comparision_alg1_3}). Also, for the received theoretical estimates, we can see that \texttt{ISTM} gives a better estimate of a solution when $p = 2$ and for not big values of $a$. When we decrease the value of $p$ and increase $a$, \texttt{AIM} will give better estimates. Also, note that the estimate \eqref{adaptive_estimate} is without any distortion, and it is adaptive where we always can implement it for any smooth objective function without any additional information about the parameters such that $L$. 

Also, from Fig. \ref{res_alg1_alg3_a_5_15}, we can see the efficiency of the proposed Algorithm \ref{adaptive_alg_L_p} (\texttt{AIM} with variable $p$), whereby this algorithm we can obtain a solution of the problem with a better estimate. 

More additional experiments, for Algorithms \ref{alg_STM_relative}, \ref{adaptive_alg} and \ref{adaptive_alg_L_p}, can be found in Subsection \ref{sec_additional_experiments}. 
\section{Conclusions and future plans}
    In this work for convex functions, we showed that \texttt{ISTM} convergences to $\hat{\varepsilon}^2 LR_0^2$ plateau with rate $\sim \frac{1}{N^{p}}$ for any intermediate parameter $p \in [1,2].$ Moreover, for strongly convex functions we found a new lower bound $\hat{\varepsilon} \lesssim \left(\frac{\mu}{L}\right)^{\frac12}$ under which relative noise $\hat{\varepsilon}$ does not affect convergence rate. We obtained it using \texttt{ISTM} for convex and smooth functions and restart technique.  However, there exists a variant of \texttt{ISTM} Algorithm especially for strongly convex functions from work \cite{gorbunov2019optimal}. In the future, we plan to study it with relative noise and, probably, get new bounds on $\hat{\varepsilon}$ and better understatement of the role of intermediate parameter $p$.   
    
    With the help of PEP, we showed that our theoretical results for \texttt{ISTM} are about to be tight. In the future following the work \cite{taylor2023optimal} we plan to find the optimal Intermediate Method and coefficients for it under fixed $\varepsilon, p.$ 

    We also proposed adaptive algorithms (\texttt{AIM}, \texttt{AIM} with variable $p$) for the constrained optimization problem. We investigated adaptivity w.r.t. both Lipschitz smoothness and intermediate parameter $p$.  The adaptivity w.r.t. parameter $p$ along the iteration process, leads to better robustness and a better estimate of a solution to the problem under consideration. We left the study of the adaptive approaches for problems with strongly convex function and relative noise in the gradient as future work.

\section{Acknowledgments}
This work was supported by a grant for research centers in the field of artificial intelligence, provided by the Analytical Center for the Government of the Russian Federation in accordance with the subsidy agreement (agreement identifier 000000D730321P5Q0002) and the agreement with the Moscow Institute of Physics and Technology dated November 1, 2021 No. 70-2021-00138.



   




\bibliography{refs}
\bibliographystyle{unsrtnat}

\newpage
\section{Appendix}\label{sec:appendix}
\subsection{\texttt{ISTM} missing proofs}
In this section, we provide full proofs of technical lemmas and the main theorems. First of all, the following technical lemma allows us to obtain estimates for the size of parameter $A_k$.
\begin{lemma}
Consider two sequences of non-negative numbers $\{\alpha_k\}_{k \geq 0}$ and $ \{A_k\}_{k \geq 0}$ such that
\begin{equation} \label{alpha_general}
\alpha_0=A_0=0, \quad A_{k+1}=A_k+\alpha_{k+1}, \quad \alpha_{k+1}=\frac{(k+2)^{p-1}}{2 a L} \quad \forall k \geq 0,
\end{equation}
where $ a \geq 1, L > 0$. Then for all $k\geq 0$
\begin{eqnarray} 
    A_{k+1} &\geq &aL \alpha_{k+1}^2, \label{ineq_A1}\\
    A_{k+1} &\geq &  \frac{(k+2)^p}{4aL}. \label{Ak_general}
\end{eqnarray}
For $p = 2$, the explicit formula of $A_{k+1}$ looks like
\begin{equation}\label{eq_20}
A_{k+1}=\frac{(k+1)(k+4)}{4 a L}.
\end{equation}
\label{main_lemma1}
\end{lemma}
\begin{proof}
With $\alpha_0 = A_0 = 0$, we write down $A_{k+1}$ as 
\begin{eqnarray*}
A_{k+1} &=& \sum_{m = 1}^{k+1} \alpha_m = \frac{1}{2aL} \sum_{m = 1}^{k+1} (m+1)^{p-1} \\
&\geq& \frac{1}{2aL} \int_0^{k+1} (x+1)^{p-1} dx = \frac{(k+2)^p}{2aL p} - \frac{1}{2aLp}.
\end{eqnarray*}
Because of $1 \leq p \leq 2$, we have 
\begin{equation*}\label{Ak_general_p_2}
    A_{k+1} \geq \frac{(k+2)^p}{4aL} + \frac{1}{2aL} \geq  \frac{(k+2)^p}{4aL}, \quad \forall k \geq 1.
\end{equation*}
We notice that $\forall p \in [1, 2]: p \geq 2p - 2$. Thus, due to \eqref{Ak_general}, we get  
$$
A_{k+1} \geq \frac{(k+2)^p}{4aL} \geq a L \frac{(k+2)^{2p-2}}{4a^2 L^2} = a L \alpha_{k+1}^2, \quad \forall k \geq 0.
$$
Finally, for $p=2$, the explicit formula looks like $A_{k+1} = \frac{1}{2aL} \sum_{m = 1}^{k+1} (m+1) $. Using  fact that $\sum_{m=1}^{k+1} (m+1) = \frac{(k+1)(k+4)}{2}$ we complete the proof.
\end{proof}

The next lemma provides details of the general convergence of Algorithm \ref{alg_STM_relative} with arbitrary coefficients $\{A_k\}, \{\alpha_k\}$. 
\begin{lemma} \cite{gorbunov2020stochastic} \label{main_lemma2}
Let $f$ be convex function (Assumption \ref{as:cvx}) and $L$-smooth (Assumption \ref{as:Lsmooth}). Let  the stepsize parameter $a$ satisfies $ a \geq 1$ and for coefficients $\{A_k\}, \{\alpha_k\}$ inequality \eqref{ineq_A1} holds true. Then after $N \geq 0$ iterations of Algorithm \ref{alg_STM_relative} with coefficients $\{A_k\}, \{\alpha_k\}$ for all $z \in \mathbb{R}^n$ we have

\begin{equation*}\label{estimate_theo}
\begin{aligned}
A_N(f(y^N)-f(z)) \leq & \frac{1}{2}\left\|z^0-z\right\|_2^2-\frac{1}{2}\left\|z^N-z\right\|_2^2+\sum_{k=0}^{N-1} \alpha_{k+1}\left\langle\theta_{k+1}, z-z^k\right\rangle \\
& +\sum_{k=0}^{N-1} \alpha_{k+1}^2 \left\|\theta_{k+1} \right\|_2^2 + \sum_{k=0}^{N-1} \alpha_{k+1}^2\left\langle\theta_{k+1}, \nabla f(x^{k+1})\right\rangle,
\end{aligned}
\end{equation*}

where 
\begin{equation*}\label{eq_theta}
    \theta_{k+1} : = \widetilde{\nabla} f(x^{k+1}) - \nabla f(x^{k+1}).
\end{equation*}
\end{lemma}

\bigskip

\noindent
\textbf{Full proof of Theorem \ref{theo:coveregence_alg1}}

Firstly, we complete the proof in a particular case, when  $p=2$ (which is the classical Similar Triangle Method (STM)), and then generalize it for all $p \in [1,2].$

\noindent
\textbf{Classic STM.}

From Lemma \ref{main_lemma2}, for any $N\geq 0$, we have 
\begin{equation}\label{eq_62}
\begin{aligned}
A_N(f(y^N)-f(x^*)) \leq & \frac{1}{2}\left\|z^0-x^*\right\|_2^2-\frac{1}{2}\left\|z^N-x^*\right\|_2^2 
\\& +\sum_{k=0}^{N-1} \alpha_{k+1}\left\langle\theta_{k+1}, x^*-z^k\right\rangle +\sum_{k=0}^{N-1} \alpha_{k+1}^2 \left\|\theta_{k+1} \right\|_2^2
\\&
 + \sum_{k=0}^{N-1} \alpha_{k+1}^2\left\langle\theta_{k+1}, \nabla f(x^{k+1})\right\rangle.
\end{aligned}
\end{equation}
Taking into account that $f(y^N) - f(x^*) \geq 0$, for all  $y^N$, and using the notation $R_k := \|z^k - x^*\|_2$, we find that for all $0 \leq k \leq N$
\begin{equation}\label{eq_64}
\begin{aligned}
    R_k^2 & \leq R_0^2+2 \sum_{l=0}^{k-1} \alpha_{l+1}\left\langle\theta_{l+1}, x^*-z^l\right\rangle+2 \sum_{l=0}^{k-1} \alpha_{l+1}^2\left\langle\theta_{l+1}, \nabla f \left( x^{l+1} \right)\right\rangle 
    \\& \;\;\;\; +2 \sum_{l=0}^{k-1} \alpha_{l+1}^2\left\|\theta_{l+1}\right\|_2^2.
\end{aligned}
\end{equation}
Now, let us prove by the induction, that the inequalities
\begin{equation}\label{eq_65}
\begin{aligned}
 R_k^2 &\stackrel{\eqref{eq_64}}{\leq} R_0^2+2 \sum_{l=0}^{k-1} \alpha_{l+1}\left\langle\theta_{l+1}, x^*-z^l\right\rangle+2 \sum_{l=0}^{k-1} \alpha_{l+1}^2\left\langle\theta_{l+1}, \nabla f\left(x^{l+1}\right)\right\rangle \\& \;\;\;\; +2 \sum_{l=0}^{k-1} \alpha_{l+1}^2\left\|\theta_{l+1}\right\|_2^2 \leq C_1^2 R_0^2,
\end{aligned}
\end{equation}
hold for $k =  0, \ldots, N$ simultaneously, where $C_1$ is defined in \eqref{eq_29}.

For $t = 0$, the inequality \eqref{eq_65} holds since $C_1 \geq 1 $. Now, let us assume (induction hypothesis) that \eqref{eq_65} holds for some $k = T-1 \leq N - 1$  and prove it for $k = T \leq N$. 

From \eqref{eq_62}, we find that
\begin{equation*}\label{eq_66_1}
\begin{aligned}
 f(y^t) - f(x^*)& \leq \frac{1}{2A_t} \Big( R_0^2 + 2 \sum_{l=0}^{t-1} \alpha_{l+1}\left\langle\theta_{l+1}, x^*-z^l\right\rangle +
 \\& \;\;\;\; + 2\sum_{l=0}^{t-1} \alpha_{l+1}^2\left\langle\theta_{l+1}, \nabla f\left(x^{l+1}\right)\right\rangle + 2\sum_{l=0}^{t-1} \alpha_{k+1}^2\left\|\theta_{l+1}\right\|_2^2 \Big),
\end{aligned}
\end{equation*}
holds for $t =  1, \ldots, T -1$. Therefore, because of \eqref{eq_65}, the next inequality also holds true 
\begin{equation}\label{eq_66}
     f(y^t) - f(x^*) \leq \frac{C_1^2 R_0^2}{2 A_t}, \quad t = 1, \ldots, T - 1.
\end{equation}
By line $4$ of Algorithm \ref{alg_STM_relative} we have  $x^1 = \frac{1}{A_1} (A_0 y^0 + \alpha_1 z^0)$. However, $A_0 = 0$ and from line $3$ of  of Algorithm \ref{alg_STM_relative},we get
$$
A_1 = \alpha_1, \quad \text{and} \quad \alpha_1 = \frac{1}{a L} \Longrightarrow L = \frac{1}{a \alpha_1}.
$$
Consequently, $x^1 = z^0$, and we bound $\left\|\nabla f\left(x^1\right)\right\|_2$ at the first iteration as follows
\begin{equation*}
\begin{aligned}
    \left\|\nabla f(x^1)\right\|_2 & =  \left\| \left(\nabla f(z^0) - \widetilde{\nabla} f(z^0)\right) + \left(\widetilde{\nabla} f(z^0) - \nabla f(z^0) \right) + \nabla f(z^0)   \right\|_2 
    \\&   \leq 2 \left\| \widetilde{\nabla} f(z^0) - \nabla f(z^0) \right\|_2 + \left\|\nabla f(z^0)\right\|_2 
    \\& \leq (2 \hat{\varepsilon} + 1) \left\|\nabla f(z^0)\right\|_2  
    \\&  \leq (2 \hat{\varepsilon} + 1)  \left\|\nabla f(z^0) - \nabla f(x^*) \right\|_2  
    \\&  \stackrel{\eqref{eq_6}}{\leq} (2 \hat{\varepsilon} + 1) L\left\|z^0-x^*\right\|_2 =\frac{ (2 \hat{\varepsilon} + 1)   R_0}{a \alpha_1}.
\end{aligned} 
\end{equation*}
Also, for $t = 0, 1, \ldots, T -1$, we have
\begin{equation*}
\begin{aligned}
 \left\|\nabla f (x^{t+1})\right\|_2  & \;\; \leq\left\|\nabla f(x^{t+1})-\nabla f(y^t)\right\|_2 +  \left\|\nabla f(y^t) \right\|_2
 \\& 
\stackrel{\eqref{eq_6},\eqref{eq_8}}{\leq} L\left\|x^{t+1}-y^t \right\|_2 + \sqrt{2 L \left(f(y^t)-f (x^*)\right)}.
\end{aligned}
\end{equation*}
From the definition of $x^{k+1}$ (see line $4$ of Algorithm \ref{alg_STM_relative}) and since $A_{k+1} = \alpha_{k+1} + A_k$ (see line $3$ of Algorithm \ref{alg_STM_relative}), we have
$$
(\alpha_{k+1} + A_k) x^{k+1} = A_k y^k + \alpha_{k+1} z^k,
$$
i.e.
\begin{equation}\label{eq_1m2m}
    x^{k+1} - y^k = \frac{\alpha_{k+1}}{A_k} (z^k - x^{k+1}).
\end{equation}
Let $ \widetilde{R}_0 := R_0,  \widetilde{R}_{k+1} := \max \{ \widetilde{R}_k, R_{k+1}\}$. From \eqref{eq_66} and  \eqref{eq_1m2m}, and since $\alpha_{k+1} = \frac{k+2}{2 a L}$,  we get for $t = 1, \ldots, T -1$, the following
\begin{equation*}
\begin{aligned}
 \left\|\nabla f (x^{t+1})\right\|_2  &\leq \frac{\alpha_{t+1} L}{A_t}\left\|x^{t+1}-z^t\right\|_2+\sqrt{\frac{L C_1^2 R_0^2}{A_t}}  
\\& \stackrel{\eqref{eq_20}}{\leq} \frac{2 L(t+2)}{t(t+3)}\left(\left\|x^{t+1}-x^*\right\|_2+\left\|x^*-z^t\right\|_2\right)  + \frac{2LC_1 R_0 \sqrt{a}}{\sqrt{t(t+3)}}
\\&
\leq \frac{4 L(t+2) \widetilde{R}_t}{t(t+3)} + \frac{2 L C_1 R_0 \sqrt{a}}{\sqrt{t(t+3)}} 
\\& 
\stackrel{\eqref{eq_65}}{\leq} \frac{4 L(t+2) C_1 R_0 }{t(t+3)} + \frac{2 L C_1 R_0 \sqrt{a}}{\sqrt{t(t+3)}} 
\\&\; = 4 L C_1 R_0\left(\frac{t+2}{t(t+3)} + \frac{\sqrt{a}}{2\sqrt{ t(t+3)}}\right) 
\\& = \frac{2 C_1 R_0}{\alpha_{t+1}} \left(\frac{(t+2)^2}{a t(t+3)}+\frac{t+2}{2\sqrt{a t(t+3)}}\right) .
\end{aligned}
\end{equation*} 
Because of $\frac{(t+2)^2}{t(t+3)} \leq \frac{(1+2)^2}{1(1+3)}=\frac{9}{4}$,  we get, for $t = 1, \ldots, T-1$ 
\begin{equation*}
   \left\|\nabla f (x^{t+1})\right\|_2 \leq  \frac{C_1 R_0}{\alpha_{t+1}}\left(\frac{9}{2 a}+\frac{3}{2 \sqrt{a}}\right). 
\end{equation*}
Let $  \frac{9}{2 a}+\frac{3}{2 \sqrt{a}} := \frac{1}{C}$, then we have
\begin{equation*}\label{eq_12nv2}
    \left\|\nabla f (x^{t+1})\right\|_2  \leq \frac{C_1 R_0}{C \alpha_{t+1}}, \quad  \forall t = 1, \ldots, T-1. 
\end{equation*}



Now we go back to inequality
\begin{equation*}
\begin{aligned}
R_T^2 &\leq R_0^2 + \underbrace{2 \sum_{l=0}^{T-1} \alpha_{l+1}\left\langle\theta_{l+1}, x^* - z^l \right\rangle}_{\Circled{1}} + \underbrace{2 \sum_{l=0}^{T-1} \alpha_{l+1}^2\left\|\theta_{l+1}\right\|_2^2}_{\Circled{2}} 
\\& \;\;\;\;  + \underbrace{2 \sum_{l=0}^{T-1} \alpha_{l+1}^2\left\langle\theta_{l+1},\nabla f(x^{l+1}) \right\rangle}_{\Circled{3}},
\end{aligned}
\end{equation*}

and provide tight upper bounds for $\Circled{1}, \Circled{2}$ and $\Circled{3}$. 

\bigskip 

\noindent
\textbf{Upper bound for $\Circled{1}$.}
We have
\begin{equation*}
    \begin{aligned}
    \left|\Circled{1}\right|&  \leq 2 \sum_{l=0}^{T-1} \alpha_{l+1}\left\langle\theta_{l+1}, x^* - z^l \right\rangle \leq 2 \sum_{l = 0}^{T-1} \alpha_{l+1} \left\|\theta_{l+1}\right\|_2 \left\|x^* - z^l\right\|_2
     \\& \leq  2 \sum_{l = 0}^{T-1} \alpha_{l+1} \hat{\varepsilon} \left\|\nabla f(x^{l+1})\right\|_2 C_1 R_0
      \leq  2 \sum_{l=0}^{T-1} \alpha_{l+1} \hat{\varepsilon}  \frac{C_1^2 R_0^2}{C \alpha_{l+1}} 
     \stackrel{T\leq N}{\leq}   \frac{ 2 N \hat{\varepsilon} C_1^2 R_0^2}{C}. 
    \end{aligned}
\end{equation*}
Let us choose $a$ such that
\begin{equation}\label{a_for_1}
    \frac{2N \hat{\varepsilon} C_1^2}{C} \leq \frac{1}{4} \; \Longrightarrow \; \left(\frac{9}{2a} + \frac{3}{2\sqrt{a}}\right) 2 N \hat{\varepsilon} C_1^2 \leq \frac{1}{4} \; \Longrightarrow \; a \gtrapprox  \max\left\{ N\hat{\varepsilon}, N^2\hat{\varepsilon}^2 \right\},
\end{equation}
then we get the following upper bound for $\Circled{1}$
\begin{equation*}\label{eq_upper_bund_1}
2 \sum_{l=0}^{T-1} \alpha_{l+1}\left\langle\theta_{l+1}, x^* - z^l \right\rangle  \leq \frac{1}{4}R_0^2.   
\end{equation*}

\noindent
\textbf{Upper bound for $\Circled{2}$.}
We have
\begin{equation*}
    \begin{aligned}
     \left|\Circled{2}\right| & \leq 2 \sum_{l=0}^{T-1} \alpha_{l+1}^2 \left\|\theta_{l+1} \right\|_2^2  \leq 2 \sum_{l = 0}^{T-1} \alpha_{l+1}^2 \hat{\varepsilon}^2 \left\|\nabla f(x^{l+1})\right\|_2^2 
     \\& \leq  2 \sum_{l = 0}^{T-1} \alpha_{l+1}^2 \hat{\varepsilon}^2 \cdot  \frac{C_1^2 R_0^2}{C^2 \alpha_{l+1}^2} 
      \stackrel{T\leq N}{\leq} \;\;  \frac{ 2 N \hat{\varepsilon}^2 C_1^2 R_0^2}{C^2}. 
    \end{aligned}
\end{equation*}
Let us choose $a$ such that
\begin{equation}\label{a_for_2}
    \frac{2N \hat{\varepsilon}^2 C_1^2}{C^2} \leq \frac{1}{4} \; \Longrightarrow \; \left(\frac{9}{2a} + \frac{3}{2\sqrt{a}}\right)^2 2 N \hat{\varepsilon}^2 C_1^2 \leq \frac{1}{4} \; \Longrightarrow \; a  \gtrapprox
   \max\left\{N\hat{\varepsilon}^2, \sqrt{ N}\hat{\varepsilon} \right\},
\end{equation}
then we get the following upper bound for $\Circled{2}$
\begin{equation*}\label{eq_upper_bund_2}
2 \sum_{l=0}^{T-1} \alpha_{l+1}^2 \left\|\theta_{l+1} \right\|_2^2  \leq \frac{1}{4}R_0^2.   
\end{equation*}

\noindent
\textbf{Upper bound for $\Circled{3}$.}
We have
\begin{equation*}
    \begin{aligned}
     \left|\Circled{3}\right| & \leq 2 \sum_{l=0}^{T-1} \alpha_{l+1}^2\left\langle\theta_{l+1},\nabla f(x^{l+1}) \right\rangle  \leq 2 \sum_{l = 0}^{T-1} \alpha_{l+1}^2 \left\|\theta_{l+1}\right\|_2 \left\|\nabla f(x^{l+1})\right\|_2 
     \\& \leq 2 \sum_{l = 0}^{T-1} \alpha_{l+1}^2 \hat{\varepsilon} \left\|\nabla f(x^{l+1})\right\|_2^2
      \leq  2 \sum_{l = 0}^{T-1} \alpha_{l+1}^2 \hat{\varepsilon}  \,  \frac{C_1^2 R_0^2}{C^2 \alpha_{l+1}^2} 
      \stackrel{T\leq N}{\leq} \;\;  \frac{ 2 N \hat{\varepsilon} C_1^2 R_0^2}{C^2}. 
    \end{aligned}
\end{equation*}
Let us choose $a$ such that
\begin{equation}\label{a_for_3}
    \frac{2N \hat{\varepsilon} C_1^2}{C^2} \leq \frac{1}{4} \; \Longrightarrow \; \left(\frac{9}{2a} + \frac{3}{2\sqrt{a}}\right)^2 2 N \hat{\varepsilon} C_1^2 \leq \frac{1}{4} \; \Longrightarrow \; a  \gtrapprox \max\left\{\sqrt{ N\hat{\varepsilon}}, N\hat{\varepsilon} \right\},
\end{equation}
then we get the following upper bound for $\Circled{3}$
\begin{equation*}\label{eq_upper_bund_3}
2 \sum_{l=0}^{T-1} \alpha_{l+1}^2\left\langle\theta_{l+1},\nabla f(x^{l+1}) \right\rangle  \leq \frac{1}{4}R_0^2.   
\end{equation*}
Consequently, we find that
$$
\begin{aligned}
    R_T^2 & \leq R_0^2 + \frac{3}{4}R_0^2 = \frac{7}{4} R_0^2 \leq C_1^2 R_0^2, \quad \text{where} \, \, C_1 = \sqrt{2}. 
\end{aligned}
$$
As a result, for every $N \geq 0$, we have 
\begin{equation}\label{eq_Rk_upperbound}
    R_k^2 \leq 2 R_0^2, \quad \forall \, k = 0,1, \ldots, N.
\end{equation}

Now, from \eqref{eq_62}, \eqref{eq_65} and \eqref{eq_Rk_upperbound}, we have 
\begin{equation*}
\begin{aligned}
 A_N \left(f(y^N)-f(x^*)\right) &  \leq R_0^2 + 2\sum_{k=0}^{N-1} \alpha_{k+1}\left\langle\theta_{k+1}, x^*-z^k\right\rangle 
  + 2 \sum_{k=0}^{N-1} \alpha_{k+1}^2 \left\|\theta_{k+1}\right\|_2^2 
  \\& \;\;\;\; + 2\sum_{k=0}^{N-1} \alpha_{k+1}^2 \left\langle \theta_{k+1}, \nabla f(x^{k+1})\right \rangle
 \\& \leq 2 R_0^2.
\end{aligned}
\end{equation*}
Consequently, we get from Lemma \ref{main_lemma1}, the following inequality
\begin{equation*}
    f(y^N)-f(x^*) \leq \frac{2R_0^2}{A_N} = \frac{8a L R_0^2}{N(N+3)}. 
\end{equation*}
From \eqref{a_for_1}, \eqref{a_for_2} and \eqref{a_for_3}  we find that if we choose $a$, such that
\begin{equation*}
a = O\left(\max \left\{ 1, \sqrt{N \hat{\varepsilon}}, N\hat{\varepsilon}, N^2\hat{\varepsilon}^2
\right\}\right),
\end{equation*}
with $\hat{\varepsilon} \in [0, 1]$, then we get
\begin{equation*}
    f(y^N) - f(x^*) \leq O\left(\max \left\{ \frac{LR_0^2}{N^2}, \frac{\sqrt{\hat{\varepsilon}} L R_0^2}{\sqrt{N^3}}, \frac{\hat{\varepsilon} LR_0^2}{N}, \hat{\varepsilon}^2 L R_0^2 
 \right\}\right).
\end{equation*}

\noindent
\textbf{Power policy (generalization to $p \in [1,2]$).}

\noindent

For the case when $p \in [1,2]$, the Lemma \ref{main_lemma2} is still valid with  $\alpha_{k+1}, A_{k+1} $ in \eqref{alpha_general}, but the bound for $\left\|\nabla f(x^{t+1})\right\|_2$ in each step is different, which can be shown as follows.

Firstly, for $\forall t = 1, 2, \ldots, T-1$, we have 
$$
\begin{aligned}
    \left\|\nabla f(x^{t+1})\right\|_2 & \leq  \frac{\alpha_{t+1}}{A_t}2L\widetilde{R}_t + \frac{\sqrt{L} C_1R_0}{\sqrt{A_t}}
    \\& \leq 4L C_1R_0 \frac{(t+2)^{p-1}}{(t+1)^p} + \frac{2\sqrt{a} LC_1R_0}{(t+1)^{p/2}}
    \\& \leq \frac{4aLC_1R_0}{(t+1)^{p/2}} \left( \frac{(t+2)^{p-1}}{a(t+1)^{p/2}} + \frac{1}{2\sqrt{a}(t+1)^p} \right).
\end{aligned}
$$
Because of $p \in [1,2]$, we find $p-1 \leq \frac{p}{2} \leq 1$, and 
$$
\frac{(t+2)^p}{(t+1)^{p/2}} \leq \frac{3^{p-1}}{2^{p/2}} \leq \frac{3}{\sqrt{2}}, \quad \text{and} \quad \frac{1}{(t+1)^p} \leq \frac{1}{2^p} \leq \frac{1}{2}. 
$$
Consequently, we get, for $t = 1, \ldots, T-1$, the following
\begin{equation}\label{upper_general_grad}
    \left\|\nabla f(x^{t+1})\right\|_2 \leq \frac{4aLC_1R_0}{(t+1)^{p/2}} \left(\frac{3}{\sqrt{2}a} + \frac{1}{4\sqrt{a}}\right) = \frac{4aLC_1R_0}{\widetilde{C} (t+1)^{p/2}},
\end{equation}
where $\frac{1}{\widetilde{C}} := \frac{3}{\sqrt{2}a} + \frac{1}{4\sqrt{a}} $.

Now, with \eqref{alpha_general},  \eqref{upper_general_grad}, in a similar way as previously, for the upper bound for $\Circled{1}$, we have
\begin{equation*}
    \begin{aligned}
    \left|\Circled{1}\right|& \leq 2 \sum_{l=0}^{T-1} \alpha_{l+1}\left\langle\theta_{l+1}, x^* - z^l \right\rangle 
     \leq \frac{4C_1^2 R_0^2 \hat{\varepsilon}}{\widetilde{C}} \sum_{l = 0}^{T-1}\frac{(l+2)^{p-1}}{(l+1)^{p/2}} 
    \\& \leq  \frac{4C_1^2 R_0^2 \hat{\varepsilon}}{\widetilde{C}} \sum_{l = 0}^{T} l^{\frac{p}{2} - 1}
    \leq  
    \frac{4C_1^2 R_0^2 \hat{\varepsilon}}{\widetilde{C}} T  T^{\frac{p}{2} - 1}
    \; \stackrel{T\leq N}{\leq} \;\; \frac{4C_1^2 R_0^2 \hat{\varepsilon} N^{p/2}}{\widetilde{C}}.
    \end{aligned}
\end{equation*}
Let us choose $a$ such that
\begin{equation}\label{general_a_for_1}
\begin{aligned}
    \frac{4C_1^2 \hat{\varepsilon} N^{p/2}}{\widetilde{C}} \leq \frac{1}{4} \; & \Longrightarrow \; \left(\frac{3}{\sqrt{2}a} + \frac{1}{4\sqrt{a}}\right) 4C_1^2 \hat{\varepsilon} N^{p/2} \leq \frac{1}{4} 
    \\& \Longrightarrow \; a \gtrapprox  \max\left\{\hat{\varepsilon} N^{p/2} , \hat{\varepsilon}^2 N^p  \right\},
\end{aligned}
\end{equation}
then we get $\left|\Circled{1}\right|\leq \frac{R_0^2}{4}$.

\bigskip

For the upper bound for $\Circled{2}$ (with \eqref{alpha_general},  \eqref{upper_general_grad}), we have
\begin{equation*}
    \begin{aligned}
    \left|\Circled{2}\right| &\leq 2 \sum_{l=0}^{T-1} \alpha_{l+1}^2 \left\|\theta_{l+1} \right\|_2^2  \leq 2 \sum_{l = 0}^{T-1} \alpha_{l+1}^2 \hat{\varepsilon}^2 \left\|\nabla f(x^{l+1})\right\|_2^2 
     \\& \leq  2 \sum_{l = 0}^{T-1} \frac{(l+2)^{2p-2}}{4a^2 L^2} \cdot \hat{\varepsilon}^2 \frac{16 a^2 L^2 C_1^2 R_0^2}{\widetilde{C}^2 (l+1)^p}
      \leq \frac{8\hat{\varepsilon}^2 C_1^2 R_0^2 }{\widetilde{C}^2} \sum_{l = 0}^{T} \frac{(l+2)^{2p-2}}{(l+1)^p} 
      \\& \leq  \frac{8\hat{\varepsilon}^2 C_1^2 R_0^2 }{\widetilde{C}^2} \sum_{l = 0}^{T} l^{\frac{p}{2} - 1} \; \stackrel{T\leq N}{\leq} \;\; \frac{8\hat{\varepsilon}^2 C_1^2 R_0^2 N^{p/2} }{\widetilde{C}^2}.
    \end{aligned}
\end{equation*}
Let us choose $a$ such that
\begin{equation}\label{general_a_for_2}
\begin{aligned}
     \frac{8\hat{\varepsilon}^2 C_1^2 N^{p/2} }{\widetilde{C}^2} \leq \frac{1}{4} \; & \Longrightarrow \; \left(\frac{3}{\sqrt{2}a} + \frac{1}{4\sqrt{a}}\right)^2 8 C_1^2 \hat{\varepsilon}^2 N^{p/2} \leq \frac{1}{4} 
    \\& \Longrightarrow \; a \gtrapprox  \max\left\{ \hat{\varepsilon} N^{p/4} , \hat{\varepsilon}^2 N^{p/2}  \right\},
\end{aligned}
\end{equation}
then we get $\left|\Circled{2}\right|\leq \frac{R_0^2}{4}$.

\bigskip

For the upper bound for $\Circled{3}$ (with \eqref{alpha_general},  \eqref{upper_general_grad}), we have
\begin{equation*}
    \begin{aligned}
   \left|\Circled{3}\right| & \leq 2 \sum_{l=0}^{T-1} \alpha_{l+1}^2\left\langle\theta_{l+1},\nabla f(x^{l+1}) \right\rangle 
   \leq 
   2 \sum_{l = 0}^{T-1} \alpha_{l+1}^2 \hat{\varepsilon} \left\|\nabla f(x^{l+1})\right\|_2^2
     \\& \leq \frac{8\hat{\varepsilon} C_1^2 R_0^2 }{\widetilde{C}^2} \sum_{l = 0}^{T} \frac{(l+2)^{2p-2}}{(l+1)^p}  \; \stackrel{T\leq N}{\leq} \;\; \frac{8\hat{\varepsilon} C_1^2 R_0^2 N^{p/2} }{\widetilde{C}^2}.
    \end{aligned}
\end{equation*}
Let us choose $a$ such that
\begin{equation}\label{general_a_for_3}
\begin{aligned}
     \frac{8\hat{\varepsilon} C_1^2 N^{p/2} }{\widetilde{C}^2} \leq \frac{1}{4} \; & \Longrightarrow \; \left(\frac{3}{\sqrt{2}a} + \frac{1}{4\sqrt{a}}\right)^2 8 C_1^2 \hat{\varepsilon} N^{p/2} \leq \frac{1}{4} 
    \\& \Longrightarrow \; a \gtrapprox  \max\left\{\sqrt{\hat{\varepsilon}} N^{p/4} ,  \hat{\varepsilon} N^{p/2}\right\},
\end{aligned}
\end{equation}
then we get $\left|\Circled{3}\right|\leq \frac{R_0^2}{4}$.

\bigskip

Therefore, for \texttt{ISTM}, we have 
\begin{equation}
f(y^N)-f(x^*) \leq \frac{2R_0^2}{A_N} \leq \frac{8a L R_0^2}{(N+1)^p},
\end{equation}
and from \eqref{general_a_for_1}, \eqref{general_a_for_2} and \eqref{general_a_for_3}, we find that by choosing $a$, such that
\begin{equation}\label{a_general_111}
a = O\left(\max \left\{ 1,\sqrt{ \hat{\varepsilon}} N^{p/4 },  \hat{\varepsilon} N^{p/2}, \hat{\varepsilon}^2 N^p
\right\}\right),
\end{equation}
with $\hat{\varepsilon} \in [0, 1]$, and $p \in [1,2]$, we get
\begin{equation}\label{last_estim_general}
f(y^N) - f(x^*) \leq O \left(\max \left\{ \frac{LR_0^2}{N^p}, \frac{\sqrt{\hat{\varepsilon}} L R_0^2}{N^{3p/4}}, \frac{\hat{\varepsilon} LR_0^2}{N^{p/2}}, \hat{\varepsilon}^2 L R_0^2
 \right\}\right).
\end{equation}


\subsection{Additional details about PEP}\label{sec:pep_theory}
In order to find tight convergence rate for \texttt{ISTM} Algorithm \ref{alg_STM_relative}, we consider the following optimization problem with fixed $N, \hat{\varepsilon}, a$
\begin{eqnarray} 
    \max \limits_{n, f, x^0} && f(x^N) - f(x^*) \label{eq:max_problem}  \\
    && f:\R^n \rightarrow \R \text{ is  $L$ -smooth and convex}, \notag \\
    &&  0 \in \nabla f(x^*), \notag \\
    && \|x^0 - x^*\|^2_2 \leq R^2,  \notag\\
    && \|\widetilde{\nabla} f(x^k) - \nabla f(x^k)\|^2_2 \leq \hat{\varepsilon}^2 \|\nabla f(x^k)\|^2_2, \quad k = \overline{0,  N-1},  \notag\\
    && x^{k+1}, y^{k+1}, z^{k+1} = \text{ISTMstep}(x^{k}, y^k, z^k, \widetilde{\nabla} f(x^k)), \quad k = \overline{0,  N-1}.  \notag
\end{eqnarray}
In this problem, we produce maximization over an infinitely dimensional domain of $L$ -smooth and convex functions. For numerical experiments problem \eqref{eq:max_problem} can be equivalently reformulated in finite-dimensional convex optimization task. This reformulation baseline was introduced in \cite{drori2014performance} and further developed for convex and $L$-smooth optimization in \cite{taylor2017smooth}. For the problems with relative inaccuracy in the gradients, this problem is also studied in \cite{de2017worst}. We will consider maximization over the set of triplets that represent points and function's values and gradients. We denote  this index set $I = \{*, 0, 1, \dots, N\}$. The equivalent problem takes the following form

\begin{eqnarray}
    \max\limits_{\begin{subarray} k n,x^*, f^*, g^* \\  \{x^i, y^i, z^i\}_{i=0}^N \\ \{f^i, g^i, \widetilde{g}^i\}_{i=0}^N \end{subarray} } && f^N - f^*  \notag \\
    && f:\R^n \rightarrow \R \text{ is  $L$ -smooth and convex}, \label{eq:L_and_conv_cond}\\
    && f^k = f(x^k), \quad g^k \in \nabla f (x^k), \quad k = *,0, 1 \dots, N, \label{eq:equal_cond}\\
    && g^* = 0,  \notag \\
    && \|x^0 - x^*\|^2_2 \leq R^2,  \notag\\
    && \|\widetilde{g}^k - g^k\|^2_2 \leq \hat{\varepsilon}^2 \|g^k\|^2_2, \quad k = \overline{0,  N-1},  \notag\\
    && x^{k+1}, y^{k+1}, z^{k+1} = \text{ISTMstep}(x^{k}, y^k, z^k, \widetilde{g}^k), \hspace{0.15cm}  k = \overline{0,  N-1}.  \notag
\end{eqnarray}

Harsh conditions \eqref{eq:L_and_conv_cond}-\eqref{eq:equal_cond} can be replaced by finite set of conditions on points $\{x^i\}_{i \in I}$, function values $\{f^i\}_{i \in I}$ and  gradients $\{g^i\}_{i \in I}$. They allow interpolating the unique $L$-smooth and convex function $f$.
\begin{theorem} \label{theo:interpolation}\cite{taylor2017smooth}
    For the set $\{x^i, f^i, g^i\}_{i \in I}$ there exists convex (Assumption~\ref{as:cvx}) and $L$-smooth (Assumption~\ref{as:Lsmooth}) function $f$ such that for all $i \in I$ we have $g^i \in \partial f(x^i)$ and $f^i = f(x^i)$ if and only if the following condition holds for every pair of indices $i \in I$ and $j \in I$
    \begin{equation}\label{eq:interpolation_cond}
        f^i - f^j - (g^j)^\top (x_i - x_j) \geq \frac{\|g^i - g^j\|^2}{2}.
    \end{equation}
\end{theorem}

Using Theorem \ref{theo:interpolation}, we reduce conditions \eqref{eq:L_and_conv_cond}-\eqref{eq:equal_cond} to $(N+2) \times (N+1)$ quadratic \eqref{eq:interpolation_cond} type inequalities. Thus, the following problem is equivalent reformulation of problem \eqref{eq:max_problem}  
\begin{eqnarray}
    \max \limits_{\begin{subarray}k n, x^*, f^*, g^* \\\{x^i, y^i, z^i\}_{i=0}^N \\\{f^i, g^i, \widetilde{g}^i\}_{i=0}^N \end{subarray}} && f^N - f^*  \notag \\
    &&  f^i - f^j - (g^j)^\top (x_i - x_j) \geq \frac{\|g^i - g^j\|_2^2}{2},\hspace{0.15cm}  i,j = *,0,  \dots, N, \label{eq:inter_cond_1}\\
    && g^* = 0, \label{eq:inter_cond_2}\\
    && \|x^0 - x^*\|^2_2 \leq R^2,\label{eq:init_cond}\\
    && \|\widetilde{g}^k - g^k\|^2_2 \leq \hat{\varepsilon}^2 \|g^k\|^2_2, \quad k = \overline{0,  N-1}, \label{eq:noise_cond}\\
    && x^{k+1}, y^{k+1}, z^{k+1} = \text{ISTMstep}(x^{k}, y^k, z^k, \widetilde{g}^k), \hspace{0.15cm}  k = \overline{0,  N-1}. \notag
\end{eqnarray}
Points $\{x^i\}_{i=1}^N, \{y^i\}_{i=0}^N, \{z^i\}_{i=0}^N$ are linear combinations of $x_0, \{\widetilde{g}^i\}_{i=0}^N$, consequently, we maximize only $2 +2*(N+1)$ namely vectors $x^0, x^*, \{g^i\}_{i \in I},  \{\widetilde{g}^i\}_{i \in I}$ and $N+1$ scalars $\{f^i\}_{i \in I}$. Moreover, the above problem is linear w.r.t. the inner products of all possible pairs of vectors and scalars. Firstly, we define Gram matrix $G := V^\top V \in \R^{ 2(N+2) \times 2(N+2)}$, where $V = \left( x^0, x^*, \{g^i\}_{i \in I} , \{\widetilde{g}^i\}_{i \in I}\right) \in \R^{d \times 2(N+2) }$ and $\mathbf{f} = \left(f_*, f_0, \dots, f_N \right) \in \R^{N+2}$. With these notations, we can rewrite the problem in SDP form

\begin{eqnarray*}
    \max \limits_{n, G, \mathbf{f} \in \R^{N+2} } && f^N - f^* \\
    &&  \Tr(M_{i,j}G) \leq f^i - f^j, \quad i,j = *,0, 1 \dots, N\\
    && \Tr(M_* G) = 0 ,\\
    && \Tr(M_R G) \leq R^2,\\
    && \Tr(N_i G) \leq 0 , \quad i = 0,1, \dots, N - 1, \\
    && G \succeq 0,
\end{eqnarray*}
where $M_{i,j}$ represents interpolation conditions \eqref{eq:inter_cond_1}, $M_*$  --- optimal condition \eqref{eq:inter_cond_2}, $M_R$ --- initial point condition \eqref{eq:init_cond} and $N_i$ --- relative noise conditions \eqref{eq:noise_cond}. Due to the cumbersome nature of the matrices, we will not write out their explicit form, however, the algorithm for their construction can be found in the work \cite{taylor2017smooth}. Note, that in the framework PEPit \cite{goujaud2022pepit} that we used for numerical experiments, these matrices and the conversion to the SDP form are done automatically.

\subsection{\texttt{AIM} missing proofs}\label{sect_missing_proofs_AIM}

\noindent
\textbf{Proof of Lemma \ref{lemma_upper_bound_Ak_fk}}
\begin{proof}
\nikita{We prove this Lemma by induction.} For $k = 0$, we have
\begin{equation*}
\begin{aligned}
  \Psi_0^* & \;\;\;\; = \min_{x \in Q} \left\{ \Psi_0 (x) = \frac{1}{2} \|x - x^0\|_2^2 + \alpha_0 f(x^0) + \alpha_0 \left\langle \widetilde{\nabla} f(x^0), x - x^0 \right\rangle\right\}
  \\& \;\;\; \stackrel{\eqref{y0_adaptive_alg}}{\geq} \alpha_0 \left( \frac{1}{2 \alpha_0} \|y^0 - x^0\|_2^2 + f(x^0) + \left\langle \widetilde{\nabla} f(x^0), y^0 - x^0 \right\rangle \right)
  \\& \;\;\; \; =  \alpha_0 \left(f(x^0) + \left\langle \widetilde{\nabla} f(x^0), y^0 - x^0 \right\rangle + \frac{2^{i_0}L_s}{2} \|y^0 - x^0\|_2^2  \right)
  \\& \;\;\;\; \geq  \alpha_0  f(y^0) - \alpha_0 \delta_0
 = A_0 f(y^0) - \alpha_0 \delta_0 = A_0 f(y^0) - M_0.
\end{aligned}
\end{equation*}
Assume that \eqref{ineq_upper_bound_Ak_fk} is valid for certain $k-1 \geq 0$, i.e.
\begin{equation}\label{induction_k_1}
    A_{k-1} f(y^{k-1}) - M_{k-1} \leq \Psi_{k-1}^*,
\end{equation}
and let us prove that it holds for $k$. \nikita{Firstly, we write down $\Psi_k^*$ based on Algorithm \ref{adaptive_alg} steps, i.e.}

\begin{equation*}
\begin{aligned}
\Psi_k^* &= \min_{x \in Q} \Psi_k(x^k)
\stackrel{\eqref{zk_adaptive_alg}}{\geq}\Psi_k(z^k)
\stackrel{\eqref{psi_star}}{=} \Psi_{k-1}(z^k) + \alpha_k \left( f(x^k) + \left\langle \widetilde{\nabla} f(x^k), z^k - x^k \right\rangle \right)
\\& 
\stackrel{\eqref{zk_adaptive_alg}}{\geq} \Psi_{k-1}(z^{k-1}) + \frac{1}{2} \left\|z^k - z^{k-1}\right\|_2^2  + \alpha_k \left( f(x^k) + \left\langle \widetilde{\nabla} f(x^k), z^k - x^k \right\rangle \right)
\\& \stackrel{\eqref{psi_star}} {\geq}
\Psi_{k-1}^* + \frac{1}{2} \left\|z^k - z^{k-1}\right\|_2^2  + \alpha_k \left( f(x^k) + \left\langle \widetilde{\nabla} f(x^k), z^k - x^k \right\rangle \right)
\\& \stackrel{\eqref{induction_k_1}} {\geq}
A_{k-1} f(y^{k-1}) - M_{k-1} + \frac{1}{2} \left\|z^k - z^{k-1}\right\|_2^2  
\\& \;\;\;\; + \alpha_k \left( f(x^k) + \left\langle \widetilde{\nabla} f(x^k), z^k - x^k \right\rangle \right)
\\& \stackrel{\eqref{Ak_adaptive_alg}}{\geq} 
(A_k - \alpha_k) f(y^{k-1}) - M_{k-1} + \frac{1}{2} \left\|z^k - z^{k-1}\right\|_2^2  
\\& \;\;\;\; + \alpha_k \left( f(x^k) + \left\langle \widetilde{\nabla} f(x_k), z^k - x^k \right\rangle \right)
\\& = (A_k - B_k) f(y^{k-1}) - M_{k-1} + \frac{1}{2} \left\|z^k - z^{k-1}\right\|_2^2  + (B_k -\alpha_k) f(y^{k-1}) 
\\& \;\;\;\; + \alpha_k \left( f(x^k) + \left\langle \widetilde{\nabla} f(x^k), z^k - x^k \right\rangle \right)
\\& \geq (A_k - B_k) f(y^{k-1}) - M_{k-1} + \frac{1}{2} \left\|z^k - z^{k-1}\right\|_2^2 
\\& \;\;\;\;+ (B_k - \alpha_k) \left( f(x^k) + \left\langle \widetilde{\nabla} f(x^k), y^{k-1} - x^k \right\rangle \right)
\\& \;\;\;\; + \alpha_k \left( f(x^k) + \left\langle \widetilde{\nabla} f(x^k), z^k - x^k \right\rangle \right)
\\& = (A_k - B_k ) f(y^{k-1}) - M_{k-1} + \frac{1}{2} \left\|z^k - z^{k-1}\right\|_2^2  + B_k f(x^k) 
\\& \;\;\;\; + \left\langle \widetilde{\nabla} f(x^k), (B_k - \alpha_k)(y^{k-1} - x^k)+ \alpha_k(z^k - x^k) \right\rangle.
\end{aligned} 
\end{equation*}
From \eqref{xk_adaptive_alg} we get
\begin{equation}
 (B_k - \alpha_k)(y^{k-1} - x^k)+ \alpha_k(z^k - x^k)  = \alpha_k (z^k - z^{k-1}).
\end{equation}
\nikita{Thus, the next inequality holds true}

\begin{equation*}
\begin{aligned}
    \Psi_k^* &\geq (A_k - B_k ) f(y^{k-1}) - M_{k-1} + B_k f(x^k) + \alpha_k \left\langle \widetilde{\nabla} f(x^k), z^k - z^{k-1} \right\rangle
    \\& \;\;\;\; + \frac{1}{2} \left\|z^k - z^{k-1}\right\|_2^2
    \\& = (A_k - B_k ) f(y^{k-1}) - M_{k-1} + B_k \Big( f(x^k) + \frac{\alpha_k}{B_k} \left\langle \widetilde{\nabla} f(x^k), z^k - z^{k-1} \right\rangle
    \\& \;\;\;\; + \frac{1}{2 B_k} \left\|z^k - z^{k-1}\right\|_2^2  \Big).
\end{aligned}
\end{equation*}
Because of $B_k = 2^{i_k} L_{k-1} \alpha_k^2$, we have $\frac{1}{B_k} =2^{i_k} L_{k-1} \frac{\alpha_k^2}{B_k^2} $ and therefore
\begin{equation*}
\begin{aligned}
\Psi_k^* & \geq    (A_k - B_k ) f(y^{k-1}) -  M_{k-1}  + B_k \Big( f(x^k) + \frac{\alpha_k}{B_k} \left\langle \widetilde{\nabla} f(x^k), z^k - z^{k-1} \right\rangle
\\& \;\;\;\; + \frac{2^{i_k} L_{k-1}\alpha_k^2}{2 B_k^2} \left\|z^k - z^{k-1}\right\|_2^2  \Big).
\end{aligned}
\end{equation*}
\nikita{Moreover, from \eqref{xk_adaptive_alg} and \eqref{wk_adaptive_alg}}, we have
$$
\frac{\alpha_k}{B_k} \left(z^k - z^{k-1}\right) = w^k - x^k.
$$
Thus, we get
\begin{equation*}
\begin{aligned}
    \Psi_k^* & \geq  (A_k - B_k ) f(y^{k-1}) -  M_{k-1}  + B_k \Big( f(x^k) + \frac{\alpha_k}{B_k} \left\langle \widetilde{\nabla} f(x^k), w^k - z^{k-1} \right\rangle
    \\& \;\;\;\; + \frac{2^{i_k} L_{k-1}}{2} \|w^k - x^k\|_2^2  \Big)
    \\& \stackrel{\eqref{criter_out_iter}}{\geq}  (A_k - B_k ) f(y^{k-1}) -  M_{k-1}  + B_k\left(f(w^k) - \delta_k\right)
    \\& \stackrel{\eqref{yk_adaptive_alg}}{\geq} A_k f(y^k) - M_{k-1} - B_k \delta_k = A_k f(y^k) - M_k.
\end{aligned}
\end{equation*}
\end{proof}

\bigskip

\noindent
\textbf{Proof of Corollary \ref{corr_rate_adaptive}}
\begin{proof}
Let $h(x) : = \frac{1}{2}\|x- x^0\|_2^2$. From Lemma \ref{lemma_upper_bound_Ak_fk}, we have
\begin{equation*}
\begin{aligned}
    A_k f(y^k) - M_k & \leq \Psi_k^* =   \min _{x \in Q}\left\{h(x) + \sum_{j=0}^k \alpha_j\left(f (x^j) + \left\langle \widetilde{\nabla} f(x^j), x - x^j \right\rangle\right) \right\} 
    \\& \leq  h(x^*) + \sum_{j=0}^k \alpha_j\left(f (x^j) + \left\langle \widetilde{\nabla} f(x^j), x^* - x^j \right\rangle\right)
    \\& \leq h(x^*) + \sum_{j=0}^k \alpha_j f(x^*) = h(x^*) + A_k f(x^*).
\end{aligned}
\end{equation*}
i.e.
$$
A_k f(y^k) - M_k \leq h(x^*) + A_k f(x^*),
$$
and
$$
A_k \left(f(y^k) - f(x^*)\right) \leq h(x^*) + M_k = h(x^*) + \sum_{j = 0}^{k} B_j \delta_j.
$$
Therefore,
$$
f(y^k) - f(x^*) \leq \frac{h(x^*)}{A_k} + \frac{1}{A_k} \sum_{j = 0}^{k} B_j \delta_j, 
$$
which is the desired result.
\end{proof}

\bigskip

\noindent
\textbf{Proof of Theorem \ref{the:rate_adaptive_alg}}
\begin{proof}
At the first, we mention that the sequences $\{\alpha_k\}_{k \geq 0}$ and  $ \{B_k\}_{k \geq 0}$ satisfy (see Lemma 4.1 in \cite{kamzolov2021universal})
$$
0 < \alpha_{k+1} \leq B_{k+1} \leq A_{k+1}.
$$
Now, let us find a lower bound for $A_k$ and an upper bound for $B_{k}$.

For $A_k$: we have $L_k \leq 2 L, \; \forall k \geq 0$ \nikita{and as a consequence}
$$
\alpha_k = \frac{1}{L_k} \left(\frac{k + 2p}{2p} \right)^{p-1} \geq \frac{1}{2L} \left(\frac{k + 2p}{2p} \right)^{p-1},
$$
with
$$
A_k = \sum_{m = 0}^{k} \alpha_m \geq \frac{1}{2L} \sum_{m = 0}^{k} \left(\frac{m+2p}{2p}\right)^{p-1}.
$$
\nikita{However, one can make an integral bound}
$$
\sum_{m = 0}^{k} \left(\frac{m+2p}{2p}\right)^{p-1} \geq \int_{0}^{k} \left(\frac{x+2p}{2p}\right)^{p-1} dx + \alpha_0 \geq 2 \left(\frac{k+2p}{2p}\right)^p.
$$
Therefore, since $p \in [1,2]$, we get
\begin{equation}\label{upper_bound_Ak}
A_k \geq \frac{1}{L}  \left(\frac{k+2p}{2p}\right)^p \geq \frac{1}{L}  \left(\frac{k+2}{4}\right)^p.
\end{equation}

For $\sum_{i=0}^{k}B_i \delta_i$,  \nikita{we use formula of $B_i$}
$$
B_i = \alpha_i^2 L_i =\left(\frac{i + 2p}{2p}\right)^{p-1} \alpha_i,
$$
\nikita{in order to make a bound} 
\begin{equation}\label{lower_bound_Bk}
    \begin{aligned}
\sum_{i=0}^{k}B_i \delta_i & = \sum_{i=0}^{k} \left(\frac{i + 2p}{2p}\right)^{p-1} \alpha_i \delta_i
      \leq \left(\frac{k + 2p}{2p}\right)^{p-1} \left(\max_{0 \leq i \leq k} \delta_i\right)   \sum_{i=0}^{k} \alpha_i 
    \\&  =  \left(\frac{k + 2p} {2p}\right)^{p-1}  \left(\max_{0 \leq i \leq k} \delta_i\right) A_k. 
\end{aligned}
\end{equation}
From \eqref{upper_bound_Ak}, \eqref{lower_bound_Bk} and Corollary \ref{corr_rate_adaptive}, we can obtain convergence rate
$$
\begin{aligned}
    f(y^k) - f(x^*) & \leq \frac{L h(x^*)}{\left(\frac{k+2}{4}\right)^p} + \frac{1}{A_k} \left(\frac{k + 2p} {2p}\right)^{p-1}  \left(\max_{0 \leq i \leq k} \delta_i\right) A_k
    \\& \leq \frac{16 L h(x^*)}{(k+2)^p} + 2\left(\max_{0 \leq i \leq k} \delta_i\right) k^{p-1}.
\end{aligned}
$$

Assume that for the distance to the solution $x^*$ from the starting point $x^0$, we have $h(x^*) = \frac{1}{2}\|x^* - x^0\|_2^2 \leq R_0^2$, then we get the following desired result 
\begin{equation*}
\begin{aligned}
   f(y^k) - f(x^*) & \leq \frac{8 L R_0^2}{(k+2)^p} + 2\left(\max_{0 \leq i \leq k} \delta_i\right) k^{p-1}
   \\& \leq \frac{ 4  R_0^2  }{(k+2)^p}\left(\max_{0 \leq i \leq k} L_i\right) + 2\left(\max_{0 \leq i \leq k} \delta_i\right) k^{p-1}.  
\end{aligned} 
\end{equation*}
\end{proof}

\subsection{Additional PEP experiments}
In this subsection, we provide additional PEP experiments for \texttt{ISTM} in the best case when $p=2$ or standard \texttt{STM}.

\bigskip
\noindent
\textbf{Extra Experiment 1.}

 The design of the experiment is similar to Experiment 1: for each fixed $a, \varepsilon$, we look for the moment $N_{\text{pep}}(a, \hat{\varepsilon})$ at which the sequence of $\{\tau_i\}_{i=0}^{N_{\max}}$ stops to decrease. We compare numerically calculated map $N_{\text{pep}} (a, \hat{\varepsilon})$ with theoretical map $a = C^2 N^2 \hat{\varepsilon} ^2$ or $N_{\text{theory}}(a, \hat{\varepsilon}) = \frac{C \sqrt{a}}{\hat{\varepsilon}}$ that shows after which iteration \texttt{STM} starts to oscillate in theory. Fig. \ref{fig:opt_a_check} shows theoretical and numerical graphs in $a$ and $N$ axes for different levels of relative noises $\hat{\varepsilon}$ with smoothness constant $L = 1$ and initial radius $R = 1$.
 
\begin{figure}[htp]
\centering
{\includegraphics[width=7cm]{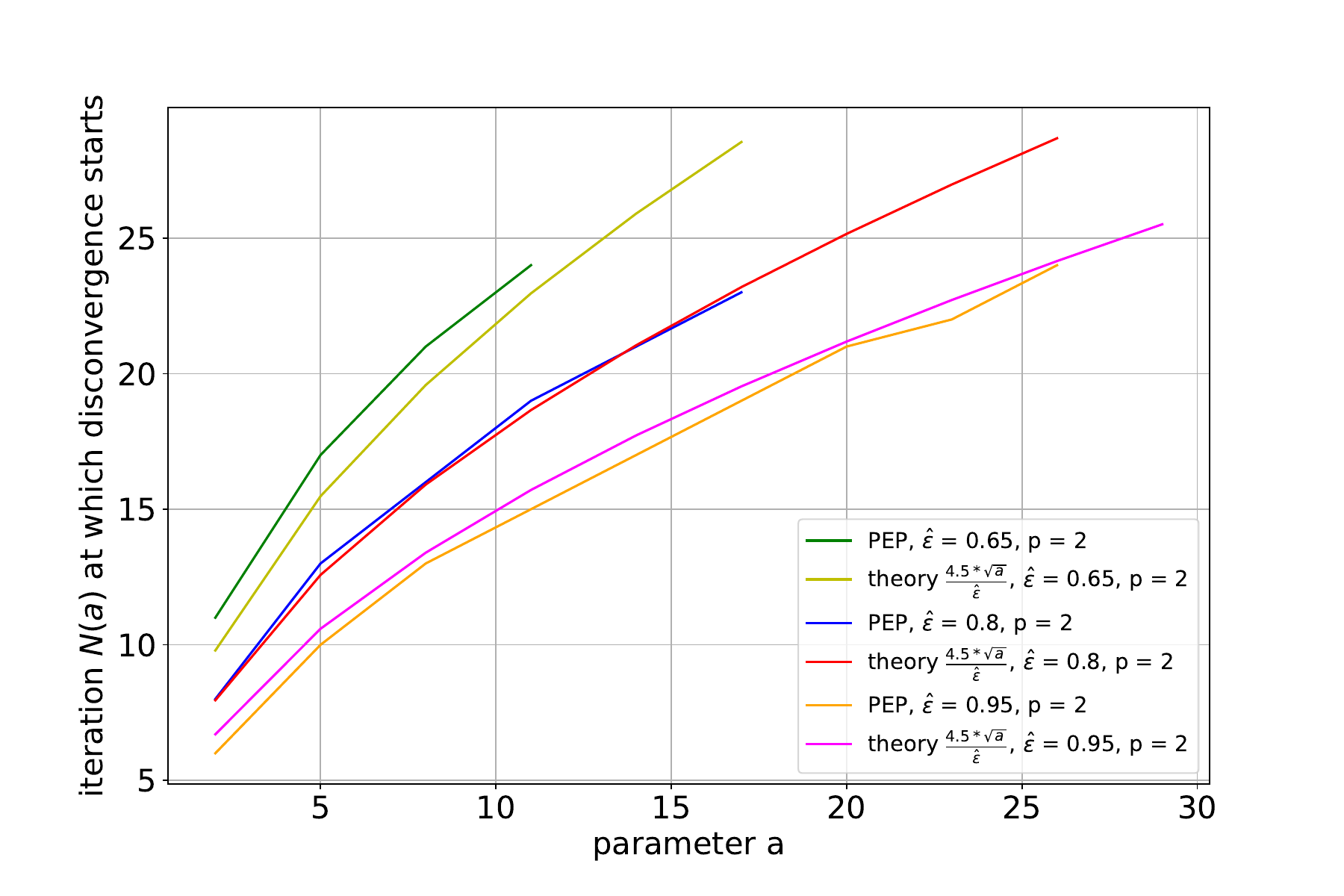} }
\caption{Comparison of theoretical and PEP calculated iteration $N$ after which \texttt{STM} with parameter $a$ starts to disconverge for different $\hat{\varepsilon}$ and $L = 1, R = 1.$ }
\label{fig:opt_a_check}
\end{figure}
As one can see PEP and theory are in excellent agreement and the constant $C \sim 4.5$ remains the same from Experiment 1.

\bigskip
\noindent
\textbf{Extra Experiment 2.}

We check the convergence formula \eqref{eq:conv_rate_alg1_proper_a} from Theorem \ref{theo:coveregence_alg1} for $p=2$,
$$
f(y^N) - f(x^*) \leq O\left(\max \left\{ \frac{LR_0^2}{N^2}, \frac{\sqrt{\hat{\varepsilon}} L R_0^2}{\sqrt{N^3}}, \frac{\hat{\varepsilon} LR_0^2}{N}, \hat{\varepsilon}^2 L R_0^2 
 \right\}\right)
 $$
 namely, the level at which the algorithm begins to oscillate after a large number of iterations $N$ with fixed $\hat{\varepsilon}$. To do this, we consider solving maximization problem \eqref{eq:max_problem} $ \tau^N = \max_f f(y^N) - f(x^*)$ setting the parameter $a$ of \texttt{STM}  to $ \max \left\{ 1, \sqrt{N \hat{\varepsilon}}, N\hat{\varepsilon}, N^2\hat{\varepsilon}^2 
 \right\}$.  The theory gives an estimate $LR_0^2\hat{\varepsilon}^2$.
Fig. \ref{fig:conv_with_opt_a} shows the dependence of the maximum difference $\tau^N$ on the number of iterations $N$ for different levels of relative noise $\hat{\varepsilon}$ with smoothness constant $L = 1$ and initial radius $R = 1$.

\begin{figure}[htp]
\centering
{\includegraphics[width=5.5cm]{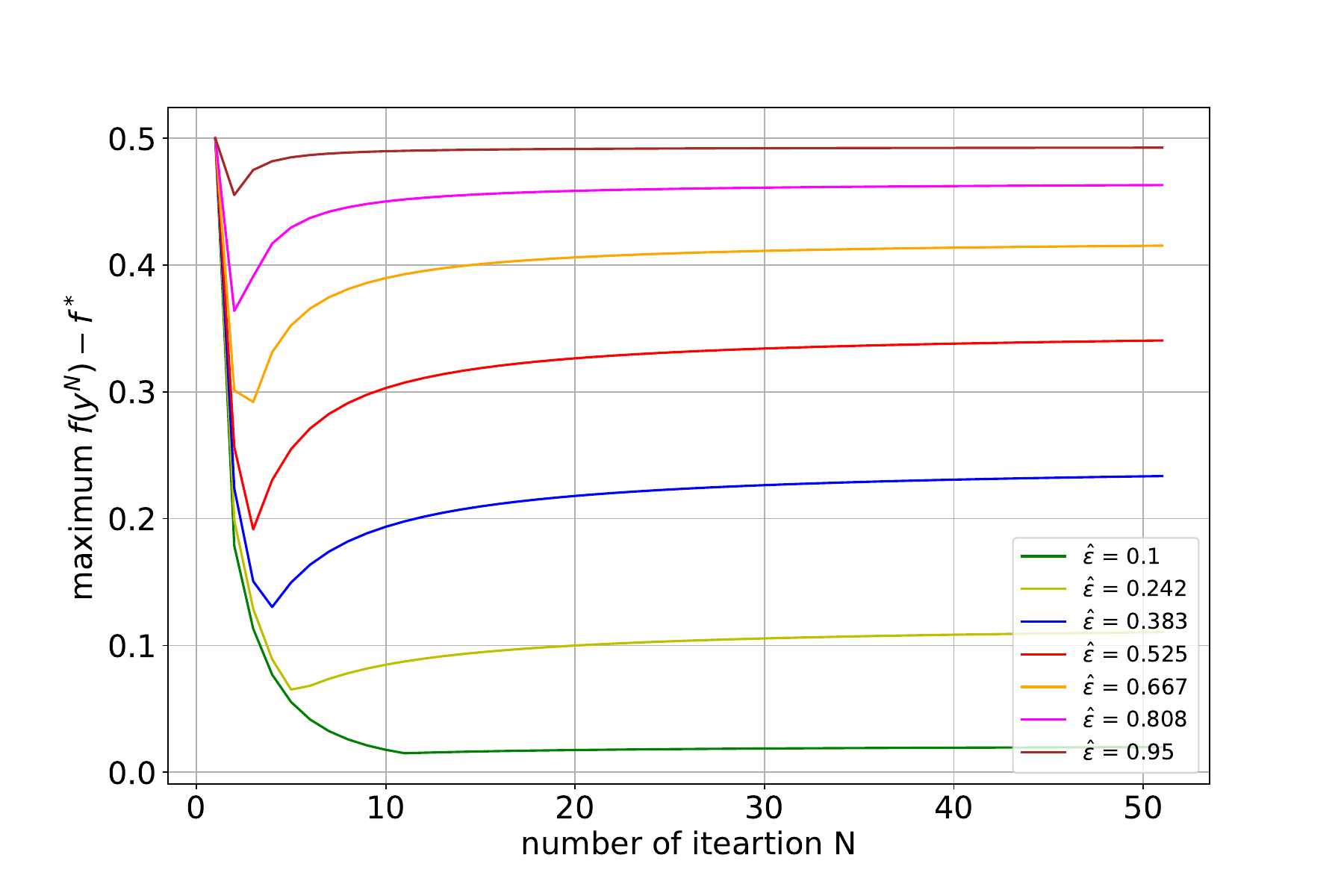} }
{\includegraphics[width=5.5cm]{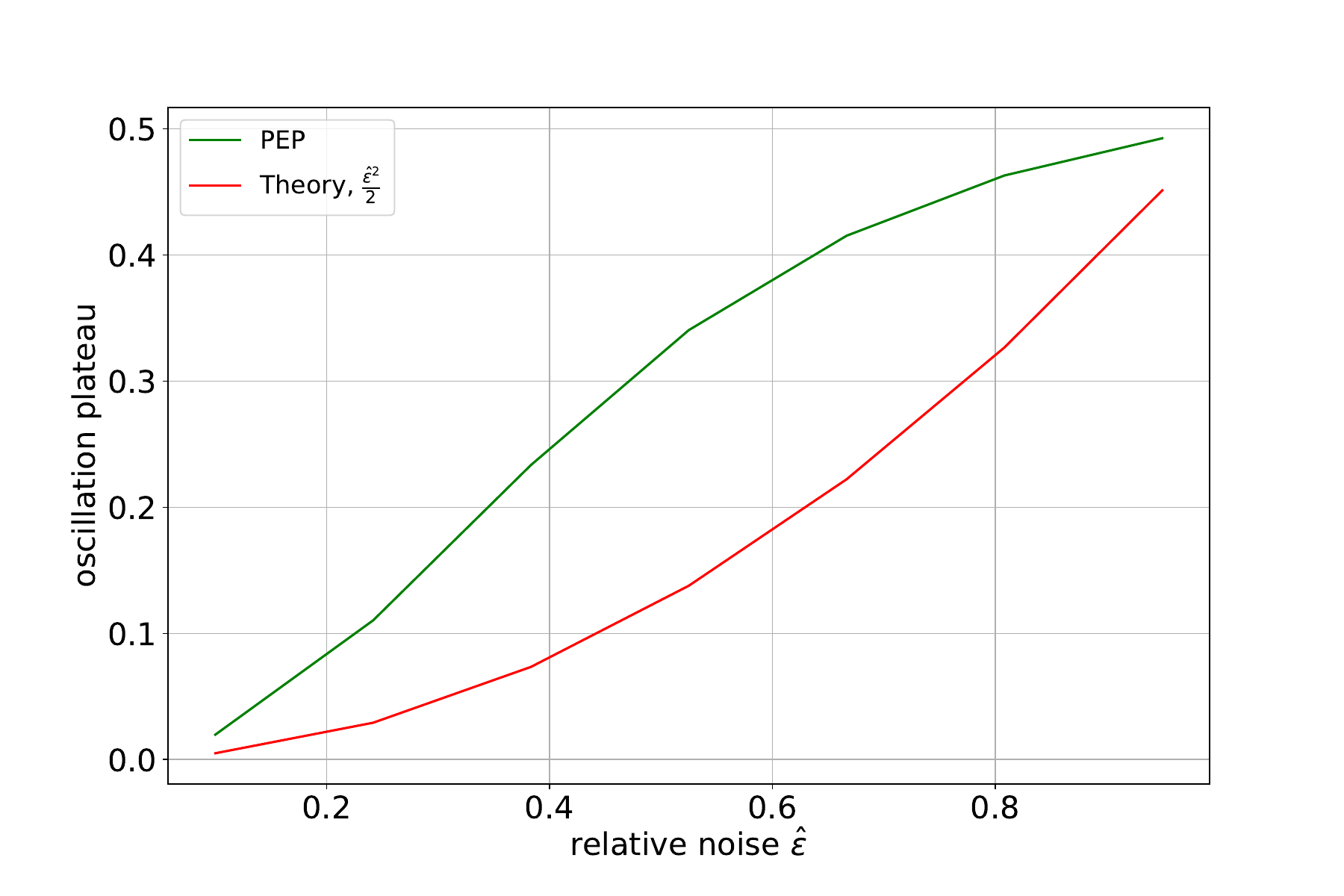} }
\caption{\textbf{Left:} Worst-case results of \texttt{STM} depending on iteration $N$ with optimal $a$ from Theorem \ref{theo:coveregence_alg1}.\textbf{Right:} Theoretical and PEP calculated plateau of oscillation for different relative noise $\hat{\varepsilon}$. }
\label{fig:conv_with_opt_a}
\end{figure}
We note that in the results of the PEP and the theory, oscillation of the algorithm around a certain plateau is visible, depending on the relative noise $\hat{\varepsilon}$. However, theoretical estimates do not quite match in order of magnitude to numerical ones.

\subsection{Additional  experiments for Algorithms \ref{alg_STM_relative}, \ref{adaptive_alg} and \ref{adaptive_alg_L_p}.}\label{sec_additional_experiments}
In this subsection, we show additional experiments connected with Algorithms \ref{alg_STM_relative} (\texttt{ISTM}), \ref{adaptive_alg} (\texttt{AIM}) and \ref{adaptive_alg_L_p} (\texttt{AIM} with variable $p$), in order to show their efficiency.  We run these algorithms with the same setting as described in Section \ref{subsec_major_experim}. 

The results of the work of Algorithm \texttt{ISTM}, represented in Fig. \ref{figs1_ISTM_a_2_10_20_30}. These results demonstrate the difference $f(y^k) - f(x^*)$ at each iteration of the algorithm, for different values of the parameter $a$. 

\begin{figure}[htp]
\centering
{\includegraphics[width=5.5cm]{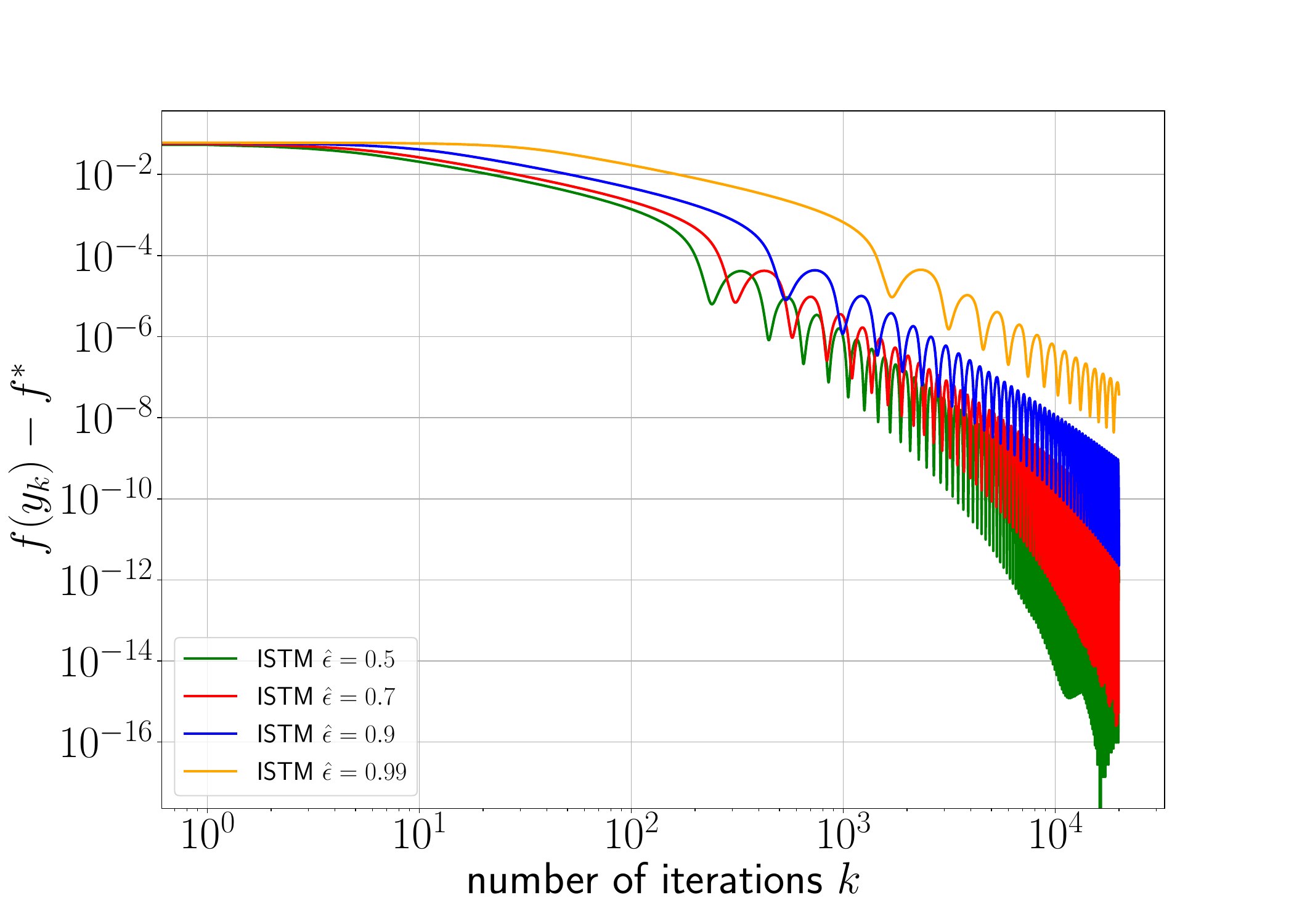} }
{\includegraphics[width=5.5cm]{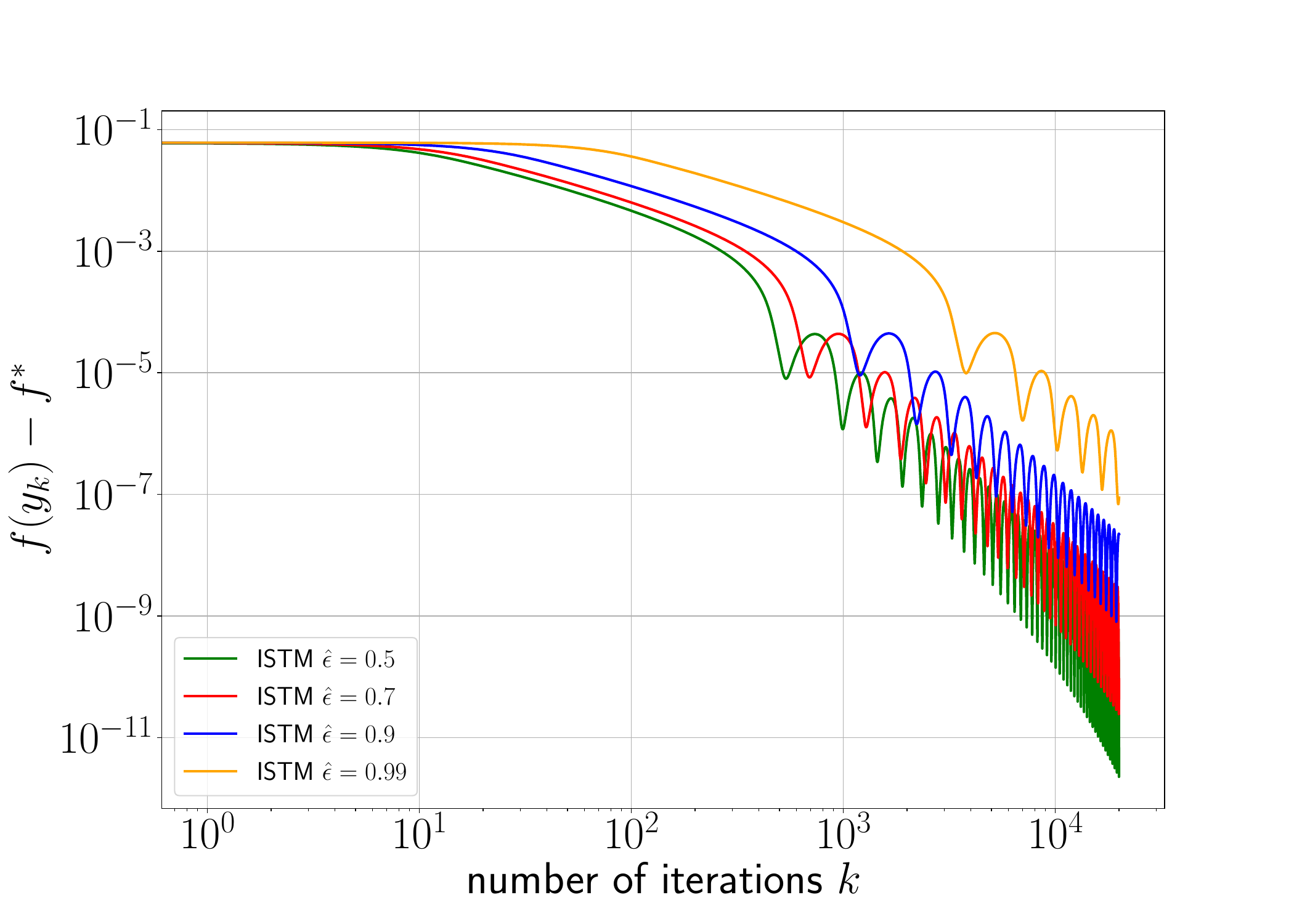}}
{\includegraphics[width=5.5cm]{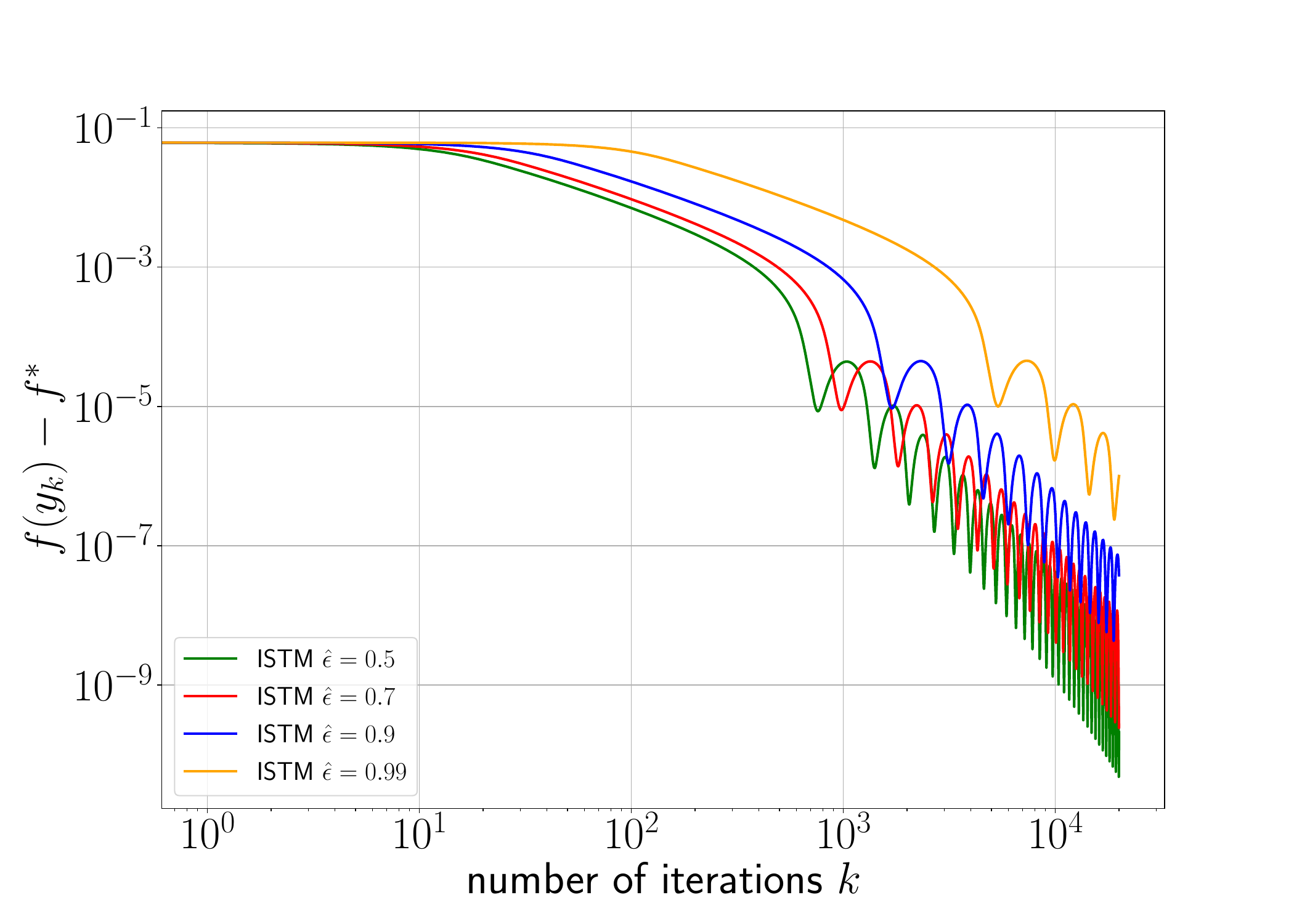} }
{\includegraphics[width=5.5cm]{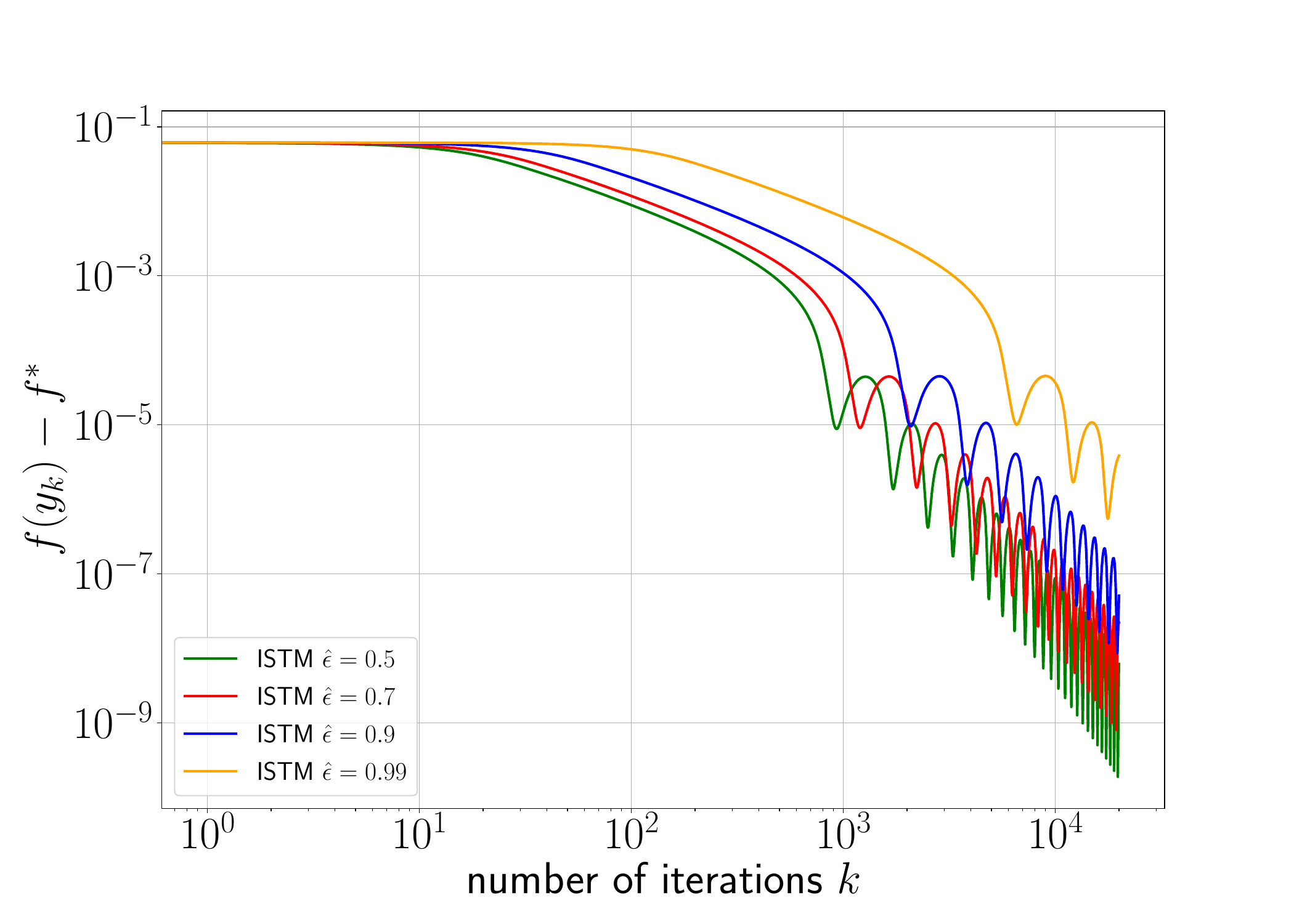} }
\caption{The results of \texttt{ISTM}
for  for  the objective function \eqref{ex_smooth}.  $a= 2$ (up-left), $a = 10$ (up-right), $a = 20$ (down-left), $a = 30$ (down-right). 
}
\label{figs1_ISTM_a_2_10_20_30}%
\end{figure}
From Fig. \ref{figs1_ISTM_a_2_10_20_30}, we can see that \texttt{ISTM}  will work better when we decrease the value of the parameter $a$. We also note that \texttt{ISTM} always converges for any value of $\hat{\varepsilon} \in [0,1]$. This fact is not satisfied for other algorithms for the class of problems under consideration with relative noise in the gradient. See \cite{vasin2023accelerated}, where it was proposed an algorithm that does not converge for  $\hat{\varepsilon} \geq 0.712$.

For Algorithms \ref{adaptive_alg} and \ref{adaptive_alg_L_p}, the results are represented in Fig. \ref{res_adaptive2_error_09_099} and Fig. \ref{res_adaptive2_adptive_estimate}. These results demonstrate the difference $f(y^k) - f(x^*)$ at each iteration and theoretical estimate \eqref{estimate2_adptive_alg}. From these figures, we also see that these algorithms always converge for any value of $\hat{\varepsilon} \in [0,1]$. Moreover, from Fig. \ref{res_adaptive2_error_09_099}, in the right, we can see the efficiency of the proposed Algorithm \ref{adaptive_alg_L_p} (\texttt{AIM} with variable $p$). 

\begin{figure}[htp]
\centering
{\includegraphics[width=5.5cm]{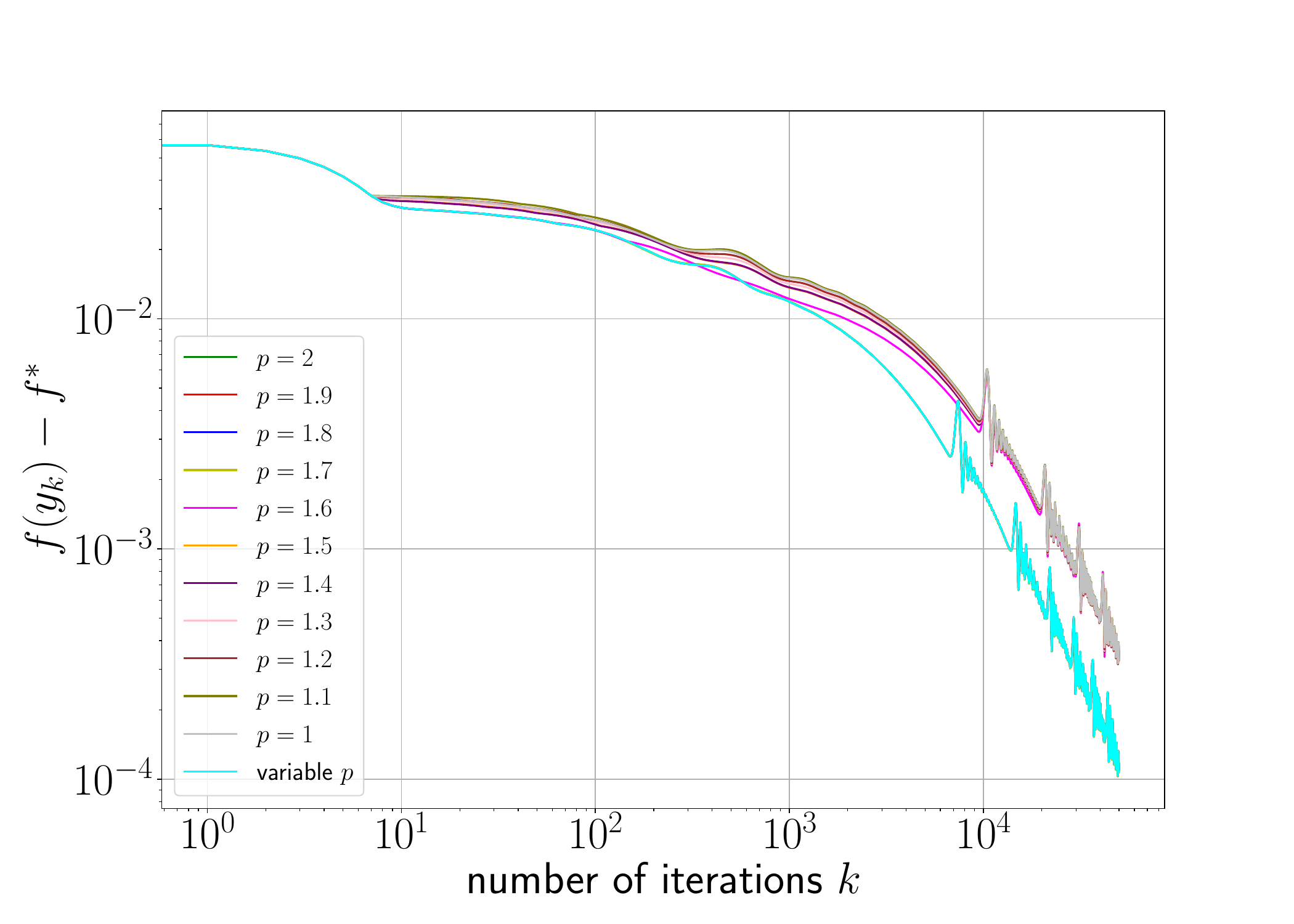} }
{\includegraphics[width=5.5cm]{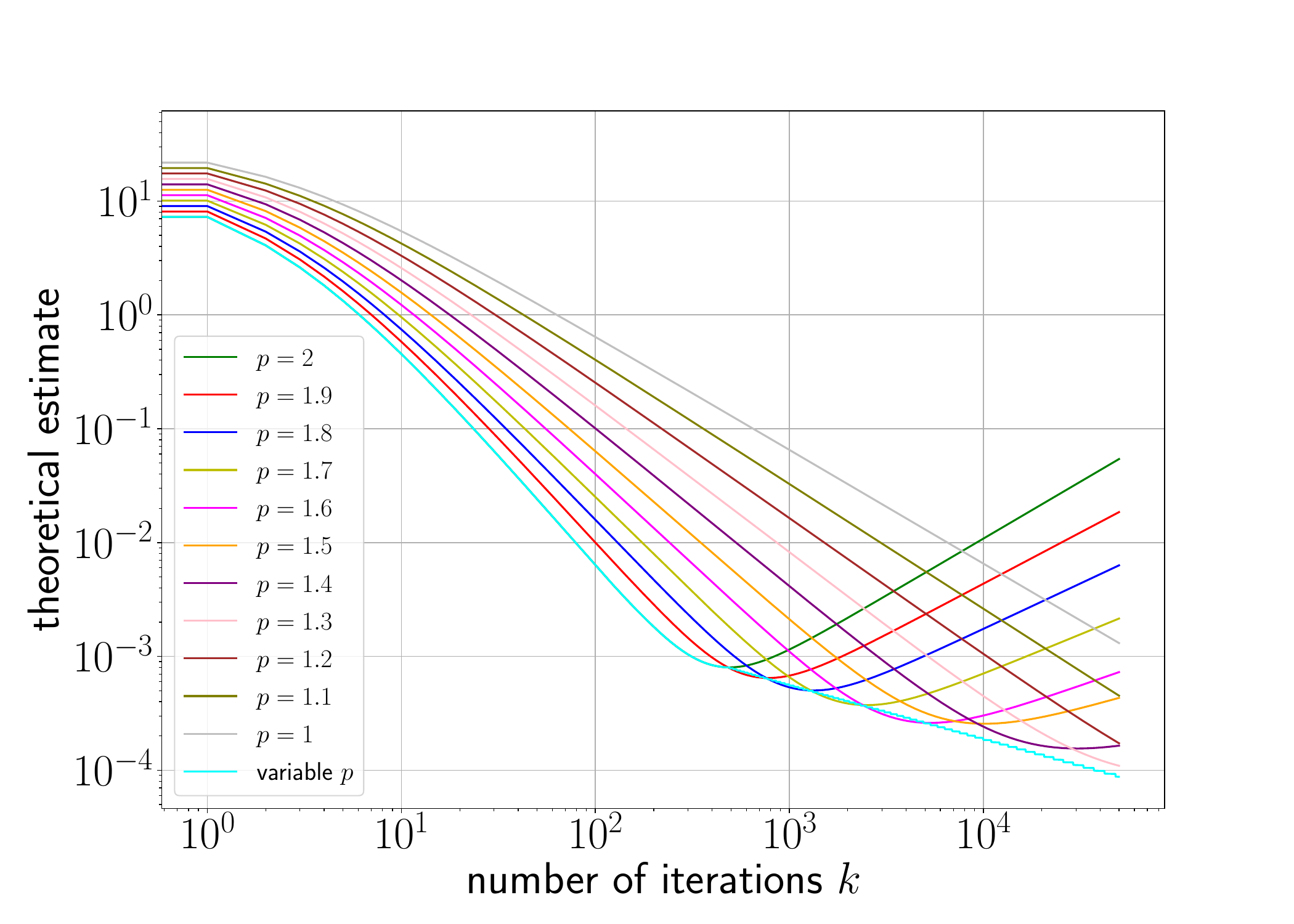} }
{\includegraphics[width=5.5cm]{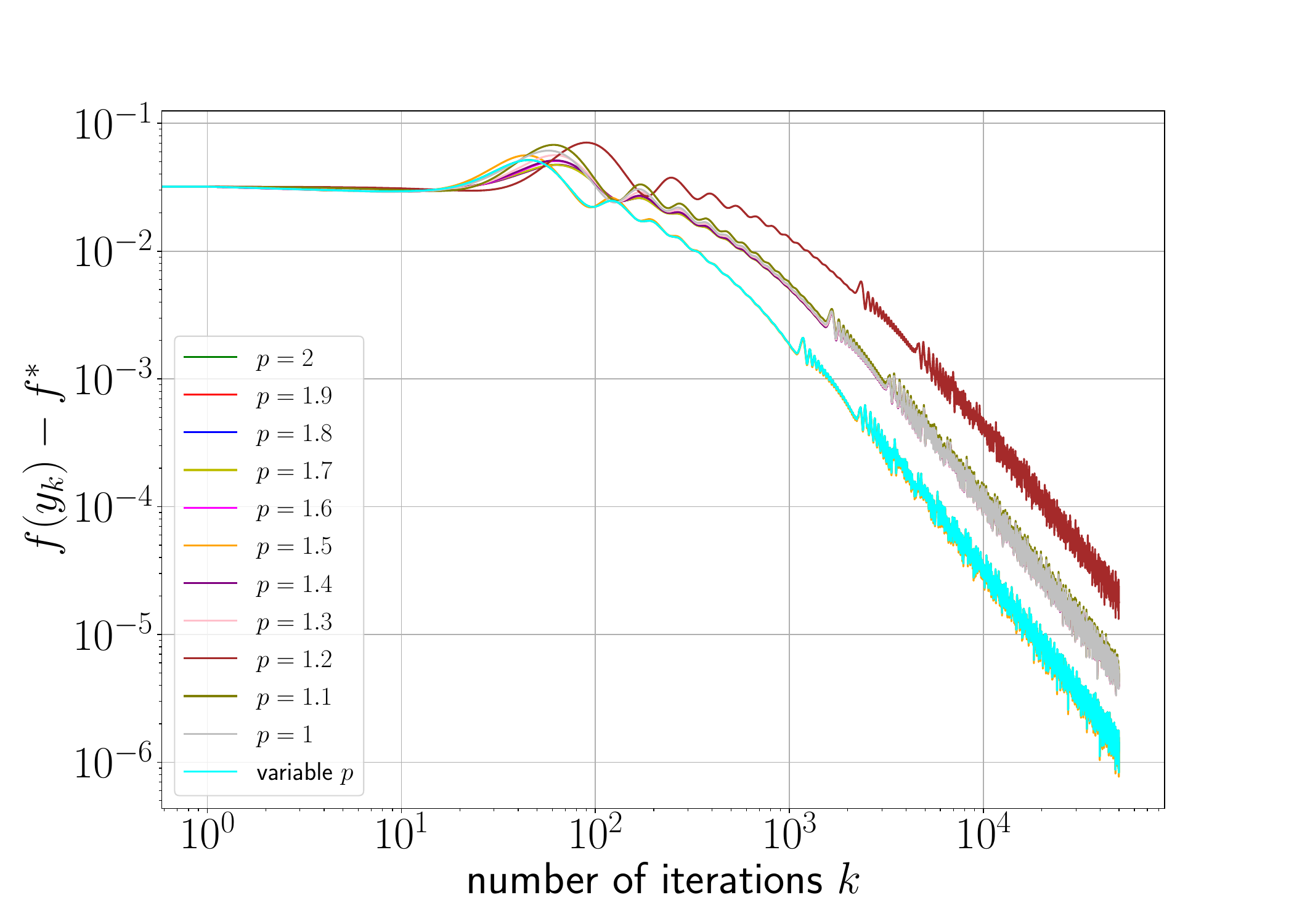} }
{\includegraphics[width=5.5cm]{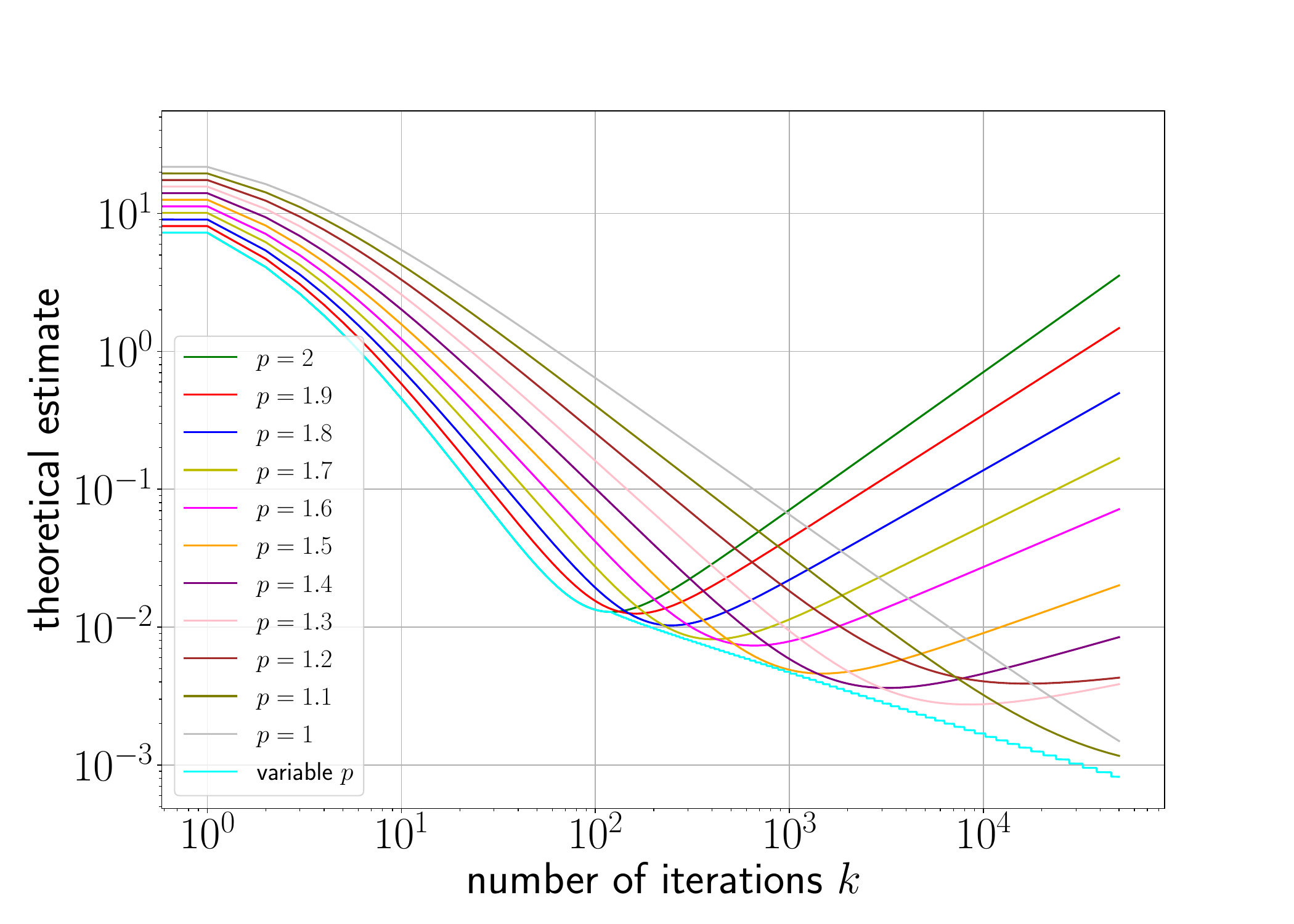} }
\caption{The results of Algorithms \ref{adaptive_alg}  and \ref{adaptive_alg_L_p}, for  the objective function \eqref{ex_smooth}, with $\hat{\varepsilon} = 0.99$ (up), and $\hat{\varepsilon} = 0.9$ (down). Here the theoretical estimate is \eqref{estimate2_adptive_alg}.}
\label{res_adaptive2_error_09_099}
\end{figure}

\begin{figure}[htp]
\centering
{\includegraphics[width=5.5cm]{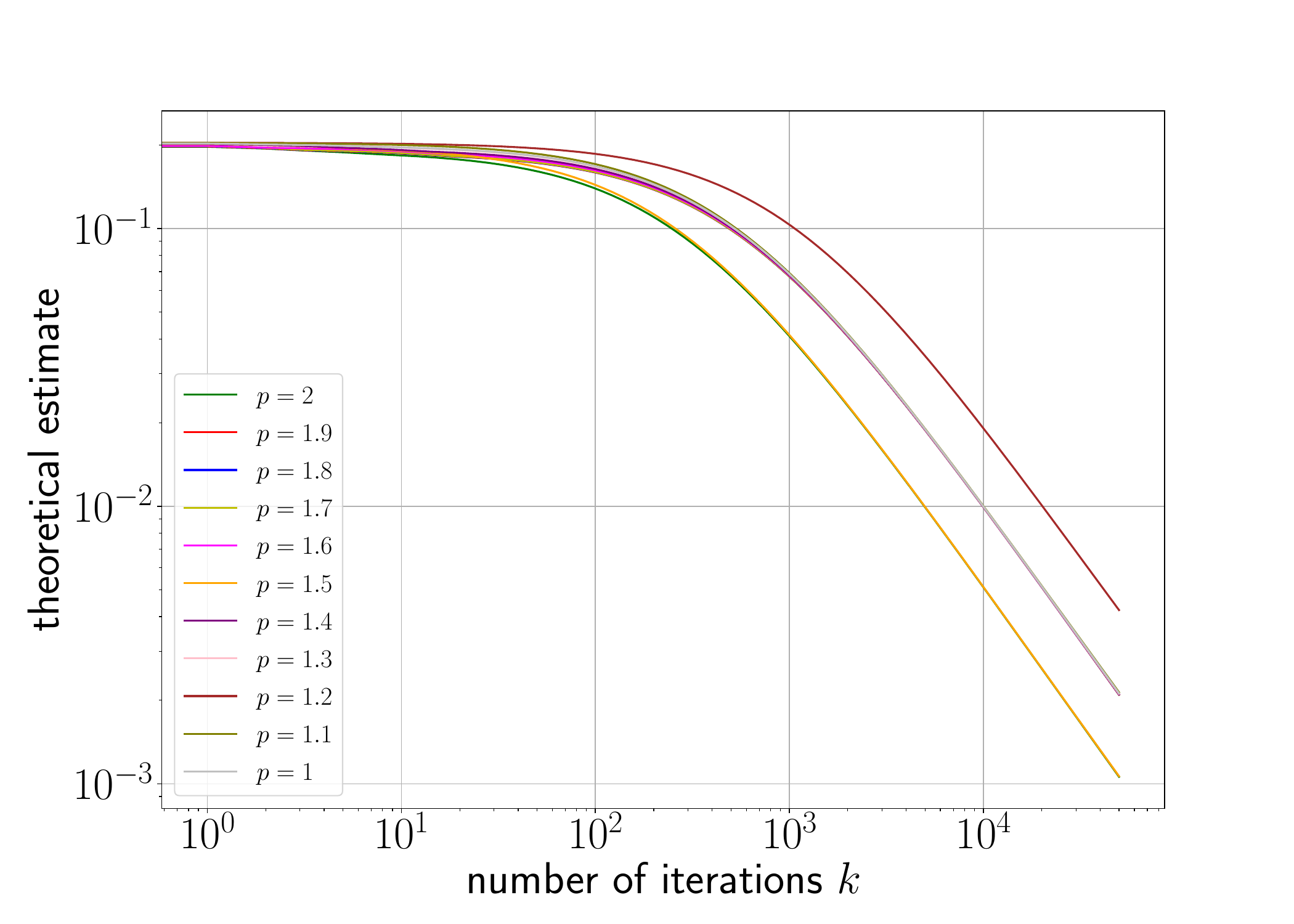} }
{\includegraphics[width=5.5cm]{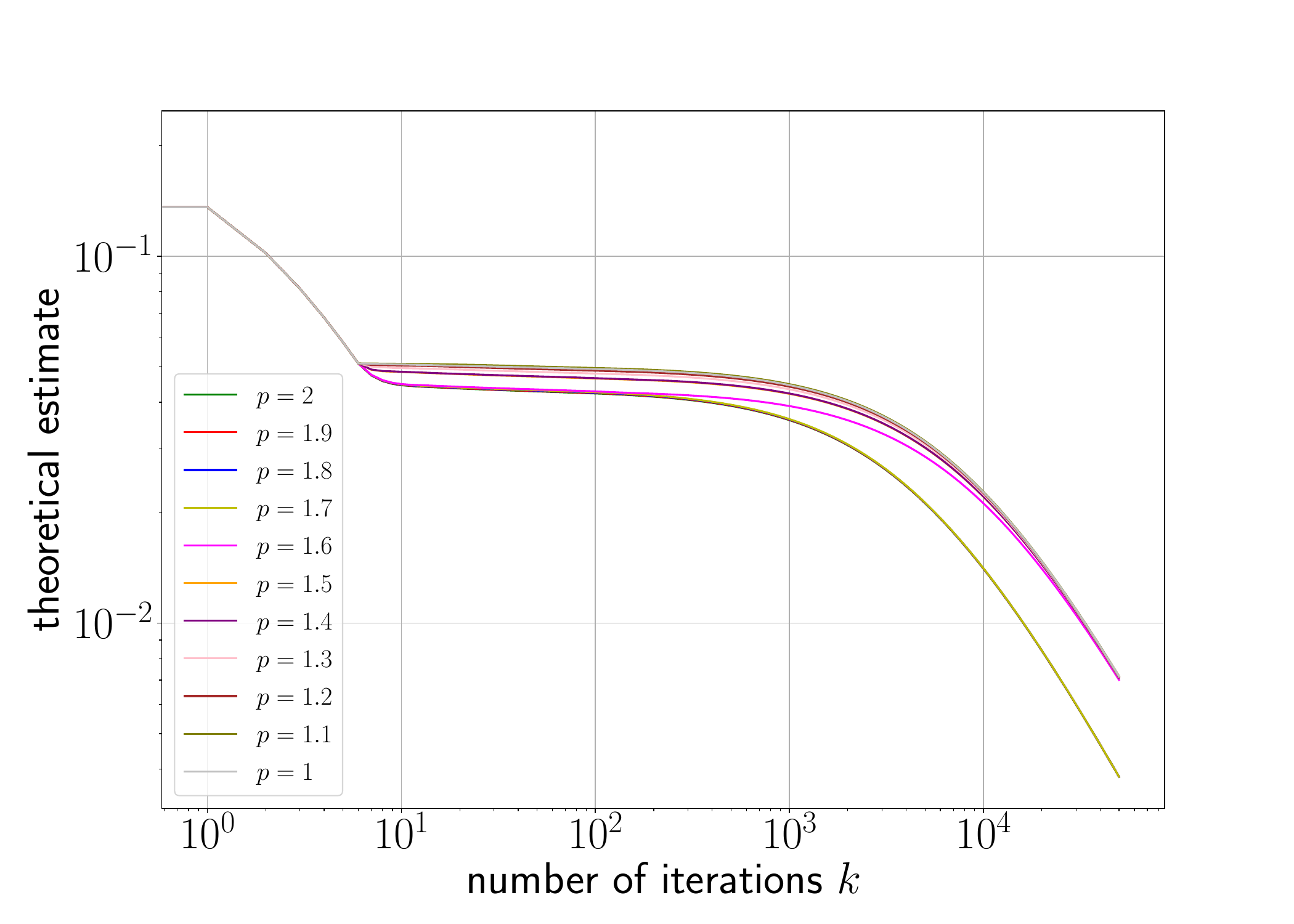} }
\caption{The results of \texttt{AIM} for  the objective function \eqref{ex_smooth}, with $\hat{\varepsilon} = 0.9 $  (left) and $\hat{\varepsilon} = 0.99$ (right). Here the theoretical estimate is \eqref{adaptive_estimate}. }
\label{res_adaptive2_adptive_estimate}
\end{figure}

In the last figure (Fig. \ref{res_comparision_alg1_3}), we see the results of the comparison between Algorithms \ref{alg_STM_relative}, \ref{adaptive_alg} and \ref{adaptive_alg_L_p}, for different values of $a$. From this figure, we can see that Algorithm \texttt{AIM},  will, at the first iterations,  work better as the value of $\hat{\varepsilon}$ decreases and the value of $a$ increases. Moreover,  the Algorithm \texttt{AIM} gives at the first iterations, a better estimate of a solution to the problem under consideration.  
 
\begin{figure}[htp]
\centering
{\includegraphics[width=5.5cm]{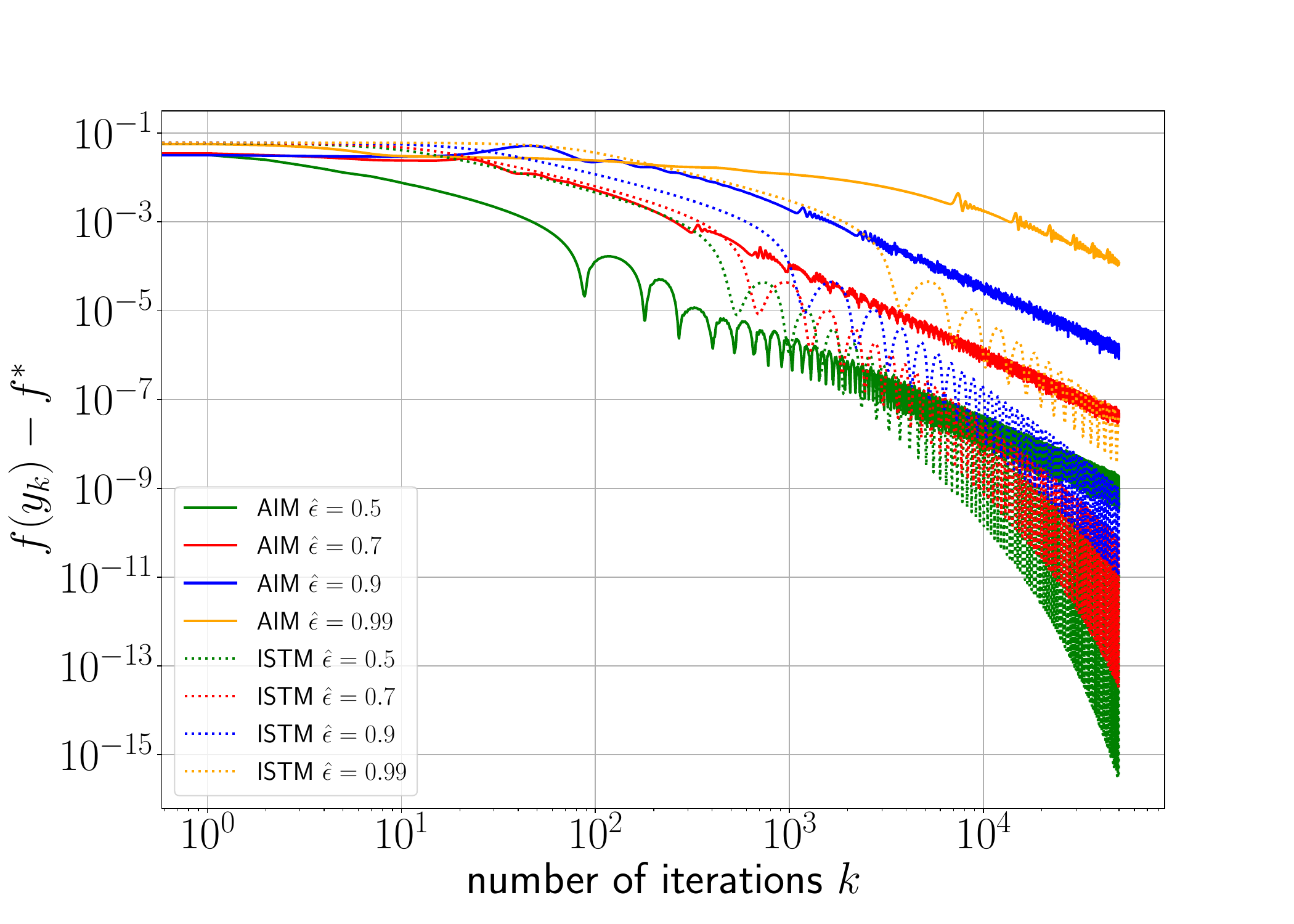} }
{\includegraphics[width=5.5cm]{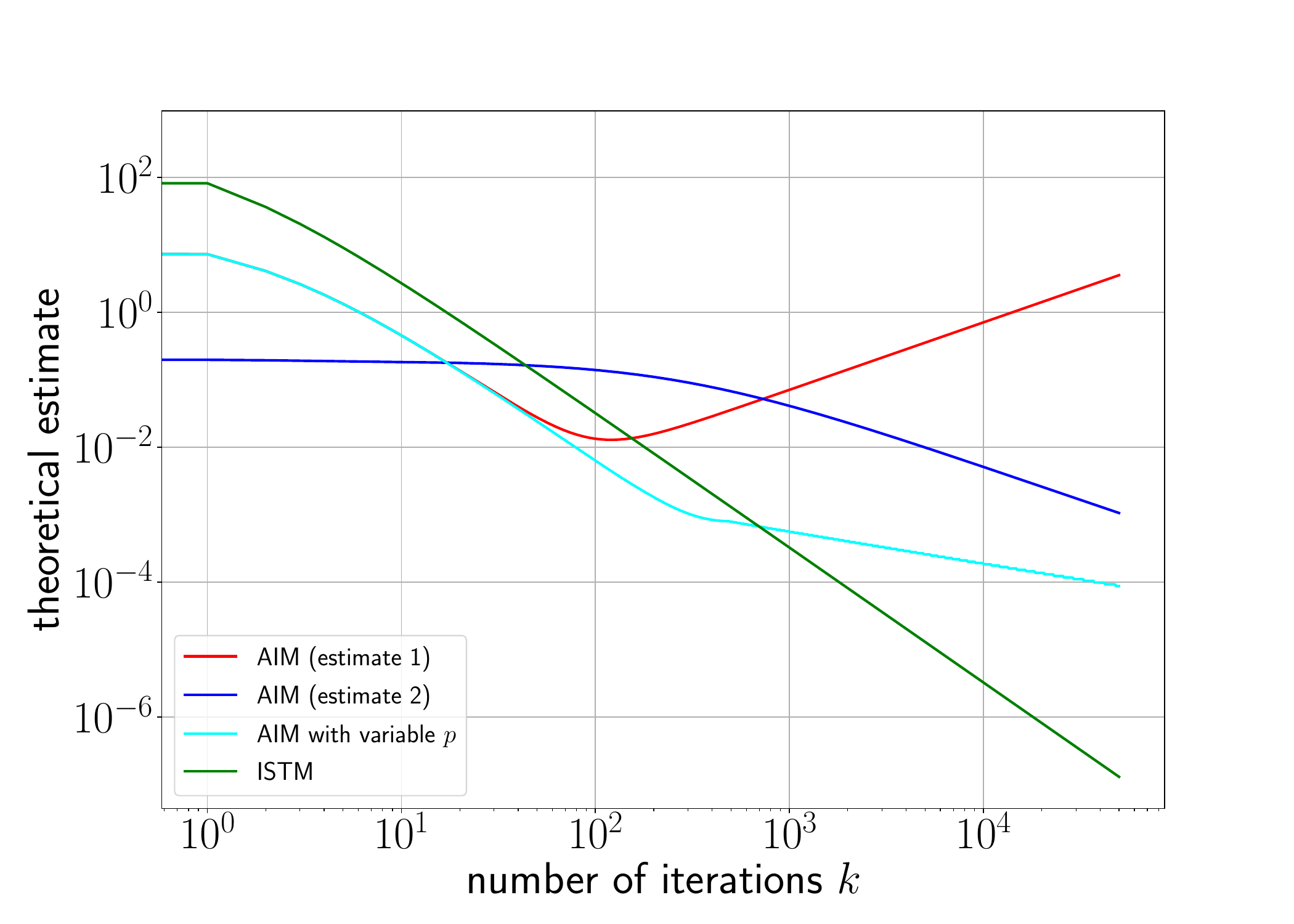} }
{\includegraphics[width=5.5cm]{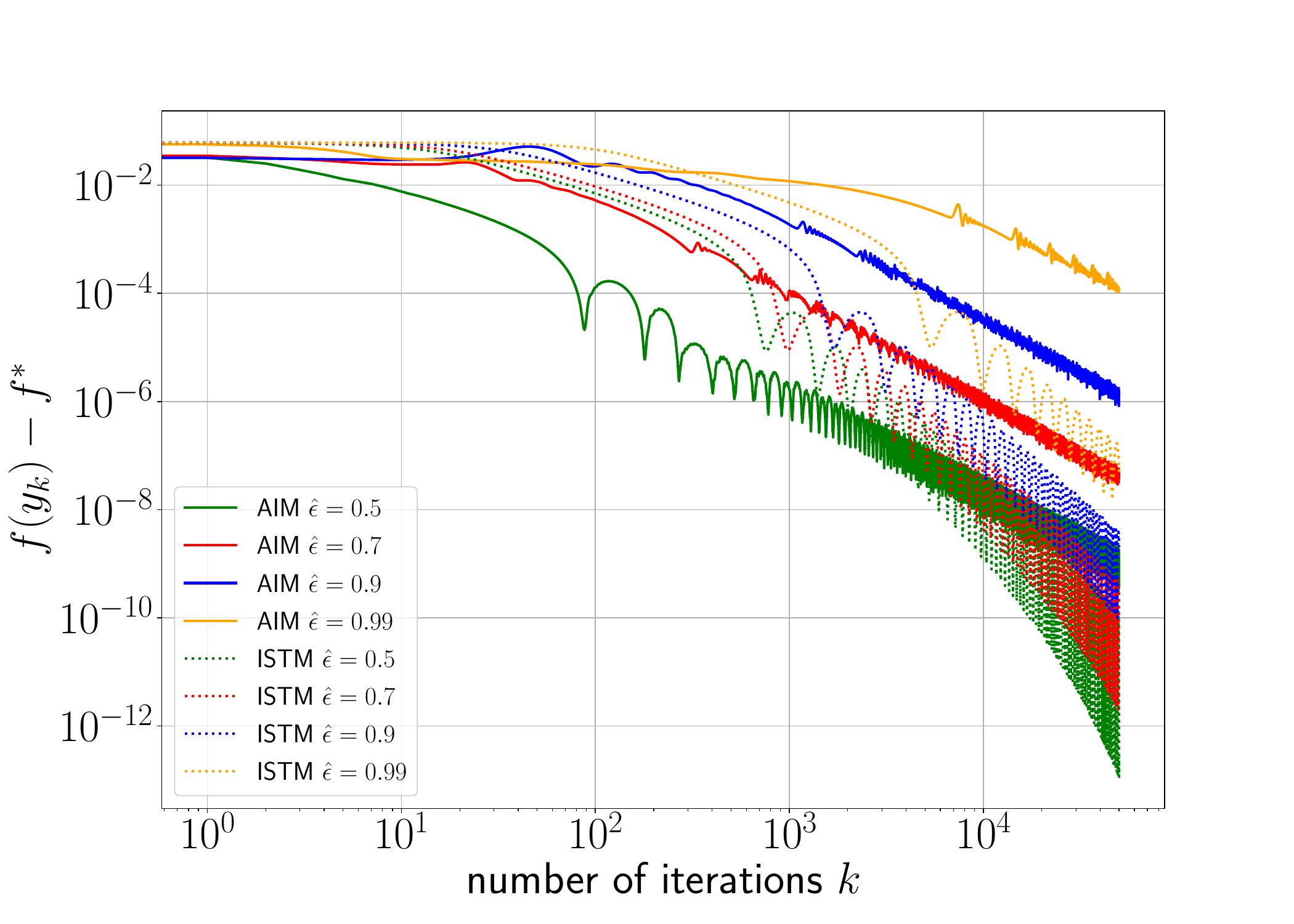} }
{\includegraphics[width=5.5cm]{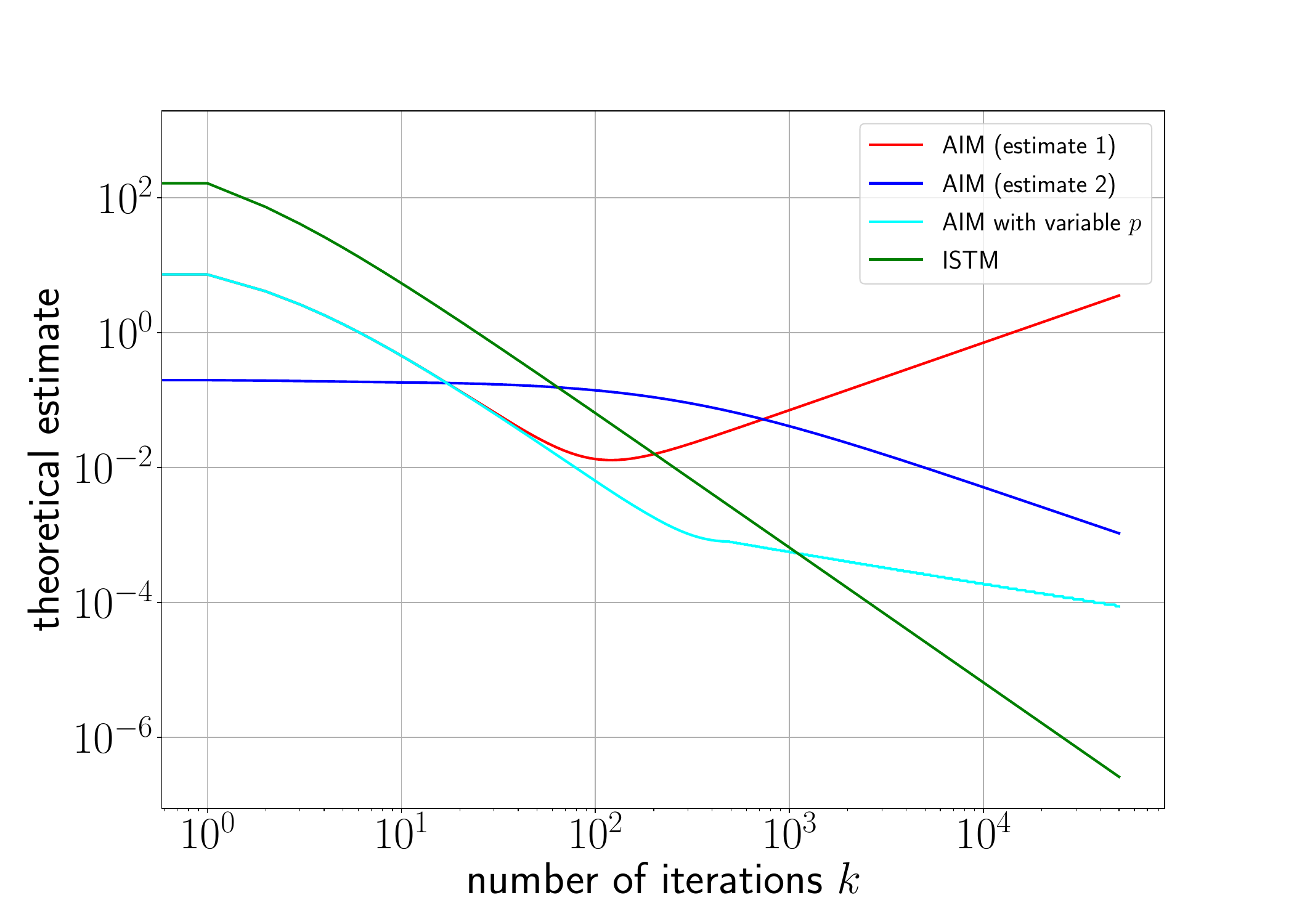} }
{\includegraphics[width=5.5cm]{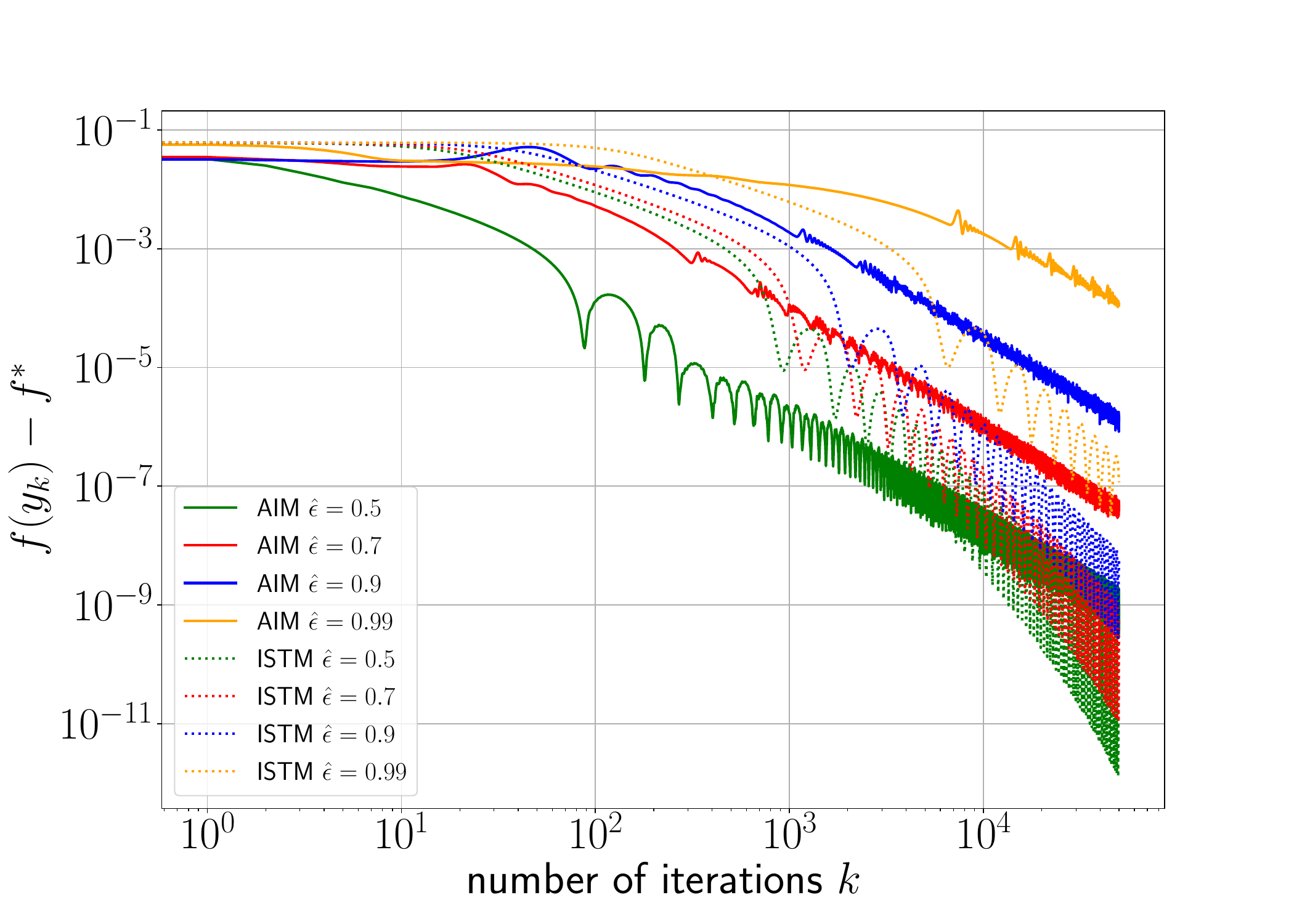} }
{\includegraphics[width=5.5cm]{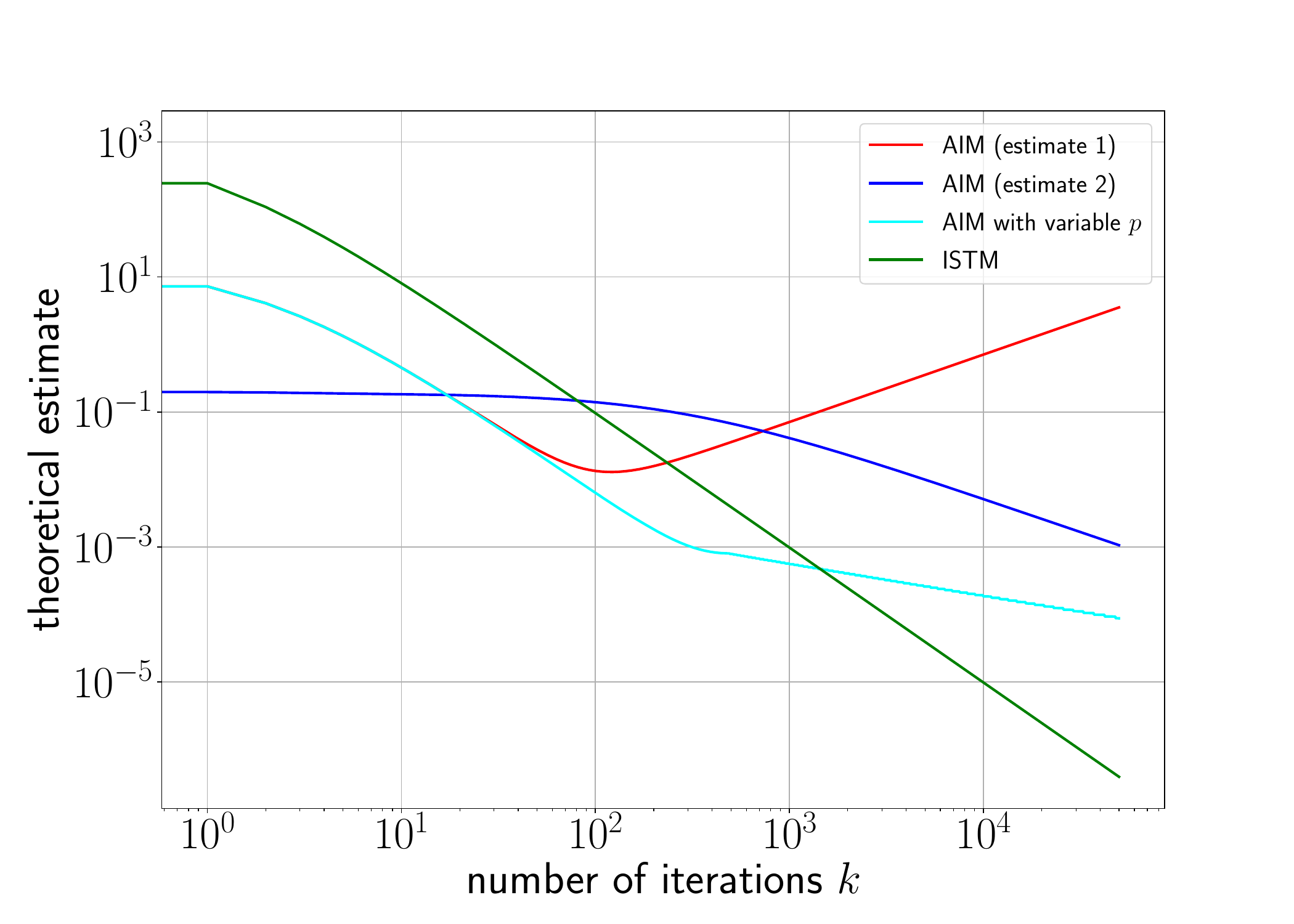} }
\caption{The results of \texttt{ISTM} (Algorithm \ref{alg_STM_relative}), \texttt{AIM} (Algorithm \ref{adaptive_alg}) and \texttt{AIM} with variable $p$ (Algorithm \ref{adaptive_alg_L_p}) for  the objective function \eqref{ex_smooth}, with $p=2$ and $a = 10, 20, 30$ from up to down. }
\label{res_comparision_alg1_3}
\end{figure}


\end{document}